\documentclass[11pt]{amsart}
 \usepackage{amsmath,amsthm,amsfonts,amssymb}
\usepackage{verbatim}
\usepackage{latexsym}

\newtheorem{theorem}{Theorem}[section]
\newtheorem{lemma}[theorem]{Lemma}
\newtheorem{proposition}[theorem]{Proposition}
\newtheorem{definition}[theorem]{Definition}
\newtheorem{corollary}[theorem]{Corollary}

\newtheorem{remark}[theorem]{Remark}
\newenvironment{prf} {{\bf Proof.}}{\hfill $\Box$}

\begin{document}
\title[Simplified vanishing moment criteria]{Simplified vanishing moment criteria for wavelets over general dilation groups, with applications to abelian and shearlet dilation groups}
\author{Hartmut F\"uhr}
\email{fuehr@matha.rwth-aachen.de}
\address{Lehrstuhl A f\"ur Mathematik, RWTH Aachen University, D-52056 Aachen}
\author{Reihaneh Raisi Tousi} 
\email[corresponding author]{raisi@um.ac.ir}
\address{Department of Pure‎ ‎Mathematics‎,
 ‎Ferdowsi University of Mashhad‎, ‎P.O.Box‎
‎1159-91775‎, ‎Mashhad‎, ‎Iran}
\begin{abstract}
We consider the coorbit theory associated associated to a square-integrable, irreducible quasi-regular representation of a semidirect product group $G = \mathbb{R}^d \rtimes H$. The existence of coorbit spaces for this very general setting has been recently established, together with concrete vanishing moment criteria for analyzing vectors and atoms that can be used in the coorbit scheme. These criteria depend on fairly technical assumptions on the dual action of the dilation group, and it is one of the chief purposes of this paper to considerably simplify these assumptions. 

We then proceed to verify the assumptions for large classes of dilation groups, in particular for all abelian dilation groups in arbitrary dimensions, as well as a class called {\em generalized shearlet dilation groups}, containing and extending all known examples of shearlet dilation groups employed in dimensions two and higher. We explain how these groups can be systematically constructed from certain commutative associative algebras of the same dimension, and give a full list, up to conjugacy, of shearing groups in dimensions three and four. In the latter case, three previously unknown groups are found.  

As a result, the existence of Banach frames consisting of compactly supported wavelets, with simultaneous convergence in a whole range of coorbit spaces, is established for all groups involved. 
\end{abstract}

 \maketitle

\noindent {\small {\bf Keywords:} square-integrable group representation; continuous wavelet transform; coorbit spaces; Banach frames; vanishing moments; shearlets; anisotropic wavelet systems}

\noindent{\small {\bf AMS Subject Classification:} 42C15; 42C40; 46E35}

\section{Introduction: Wavelet coorbit theory in higher dimensions}

Coorbit theory can be understood as a group-theoretic formalism for the description of approx\-ima\-tion-theoretic properties of building blocks arising from a unitary group action. It was initially developed with the aim to provide a unified view of large classes of function spaces, including the family of Besov spaces on one hand, with the underlying group given by the $ax+b$-group, and the modulation spaces, associated to a unitary action of the Heisenberg group as underlying group, on the other. Later it was seen to apply to other settings and groups as well, for example the shearlet groups in dimensions two and higher  \cite{DaKuStTe,DaStTe12,DaHaTe12}. This paper continues work begun in \cite{Fu_coorbit,Fu_atom}, which provided explicit criteria for wavelets associated to general dilation groups, ensuring their suitability as analyzing vectors and/or atoms in the coorbit scheme. The results in these sources depend on a number of fairly technical conditions on the dilation groups, defined as {\em (
strong) temperate embeddedness} of the 
associated open dual orbit. 
It is the chief purpose of this paper to provide simpler criteria for these conditions to hold, and to verify these 
for large classes of dilation groups. In particular, our results cover all shearlet dilation groups that have been considered so far, thus extending the known results for these groups in a unified manner. 

\subsection{Notation and preliminaries}

Before we describe the aims of this paper in more detail, let us quickly fix some notation that will be used throughout. 
We let $\mathbb{R}^+$ denote the set of strictly positive real numbers. 
$|\cdot|: \mathbb{R}^d \to
\mathbb{R}$ denotes the euclidean norm.  Given a matrix $h \in \mathbb{R}^{d \times d}$, the operator norm
of the induced linear map $(\mathbb{R}^d, |\cdot|) \to  (\mathbb{R}^d, |\cdot|)$ is denoted by
 $\| h \|$. By a slight abuse of notation we use $|\alpha| = \sum_{i=1}^d \alpha_i$ for multiindices $\alpha \in \mathbb{N}_0^d$. 
 
 Given $f \in {\rm L}^1(\mathbb{R}^d)$, its Fourier transform is defined as
\[
 \mathcal{F}(f)(\xi) := \widehat{f}(\xi) := \int_{\mathbb{R}^d} f(x) e^{-2\pi i \langle x,\xi \rangle} dx~,
\] with $\langle \cdot, \cdot \rangle$ denoting the euclidean scalar product on $\mathbb{R}^d$. 
We let $\mathcal{S}(\mathbb{R}^d)$ denote Schwartz space. Related to this space is
the family of Schwartz norms, defined for $r,m>0$ by 
\[
| f |_{r,m} = \sup_{x \in \mathbb{R}^d, |\alpha| \le r} (1+|x|)^{m}
|\partial^\alpha f (x)|~.
\] for any function $f: \mathbb{R}^d \to \mathbb{C}$ with suitably many partial derivatives. 

$\mathcal{S}'(\mathbb{R}^d)$ denotes the dual of $\mathcal{S}(\mathbb{R}^d)$, the space of tempered distributions. We denote the extension of the Fourier transform to $\mathcal{S}'(\mathbb{R}^d)$ by the same symbols as for ${\rm L}^1$-functions.  For any subspace $X \subset
\mathcal{S}'(\mathbb{R}^d)$, we let $\mathcal{F}^{-1} X$ denote its
inverse image under the Fourier transform.

In order to avoid cluttered notation, we will occasionally use
the symbol $ X \preceq Y$ between expressions $X,Y$ involving one or
more functions on or vectors in $\mathbb{R}^d$,  to indicate the existence of a constant $C>0$,
independent of the functions and vectors occurring in $X$ and $Y$, such that
 $X \le CY$. 

 Our conventions regarding locally compact groups, Haar measure etc. are the same as in Folland's book \cite{Folland_AHA}. 
We fix a closed matrix group $H < {\rm
GL}(d,\mathbb{R})$, the so-called {\bf dilation group}, and let $G =
\mathbb{R}^d \rtimes H$. This is the group of affine
mappings generated by $H$ and all translations. Elements of $G$ are
denoted by pairs $(x,h) \in \mathbb{R}^d \times H$, and the product
of two group elements is given by $(x,h)(y,g) = (x+hy,hg)$. 
The left Haar measure of $G$ is given by $d(x,h) = |\det(h)|^{-1} dx dh$, and the modular
function of $G$ is given by $\Delta_G(x,h) = \Delta_H(h)
|\det(h)|^{-1}$.
 
$G$ acts unitarily on ${\rm L}^2(\mathbb{R}^d)$ by the {\bf quasi-regular
representation} defined by
\begin{equation} \label{eqn:def_quasireg}
[\pi(x,h) f](y) = |{\rm det}(h)|^{-1/2} f\left(h^{-1}(y-x)\right)~.
\end{equation}
 
\subsection{Continuous wavelet transforms in higher dimensions}
Let us now shortly describe the ingredients going into the construction of higher-dimensional continuous wavelet transforms. For more background on the representation-theoretic aspects we refer to the book \cite{Fu_LN}, and to the previous papers \cite{Fu_coorbit,Fu_atom} for additional information concerning coorbit theory for this setting. 

We assume that the dilation group $H$ is chosen such that $\pi$ is an {\bf (irreducible) square-integrable
representation}, and call such dilation groups {\rm irreducibly admissible}. Recall that if $\pi$ is already known to be irreducible, then square-integrability means
that there exists at least one nonzero {\bf admissible vector} $\psi
\in {\rm L}^2(\mathbb{R}^d)$ such that the matrix coefficient
\[
(x,h) \mapsto \langle \psi, \pi(x,h) \psi \rangle
\] is in ${\rm L}^2(G)$, the ${\rm L}^2$-space associated
to a left Haar measure on $G$. It then follows that there is a dense subspace of admissible vectors, and for each such $\psi$, the 
associated wavelet transform
\[
\mathcal{W}_\psi : {\rm L}^2(\mathbb{R}^d) \ni f \mapsto \left(
(x,h) \mapsto \langle f, \pi(x,h) \psi \rangle \right) \in {\rm L}^2(G)
\] is a scalar multiple of an isometry, giving rise
to the weak-sense {\bf wavelet inversion formula}
\begin{equation} \label{eqn:wvlt_inv}
f = \frac{1}{c_\psi} \int_G \mathcal{W}_\psi f(x,h) \pi(x,h) \psi ~
d\mu_G(x,h)~.
\end{equation}

A thorough understanding of the properties of the wavelet transform hinges on the {\bf dual action}, i.e., 
 the right linear action $\mathbb{R}^d \times H \ni(\xi,h) \mapsto
 h^T \xi$. As a first important instance of this principle one can name the property of square-integrability itself: By the results of \cite{Fu96,Fu10}, $H$ is irreducibly admissible iff the dual action has a single open orbit
 $\mathcal{O} = \{ h^T \xi_0 : h \in H \} \subset \mathbb{R}^d$ of full measure (for some $\xi_0 \in \mathcal{O}$), such
 that in addition the stabilizer group $H_{\xi_0} = \{ h \in H : h^T \xi_0 = \xi_0 \}$ is
 compact. Note that this condition does not depend on a particular choice of $\xi_0 \in
 \mathcal{O}$, since all stabilizers of elements in $\mathcal{O}$ are conjugate. 
 The dual orbit will also be of central importance to this paper.

\subsection{A sketch of coorbit theory} 

 Let us next describe the pertinent notions from coorbit theory. In colloquial terms, coorbit theory can be described as a "theory of nice wavelets and nice signals". Nice signals are those that are well approximated by linear combinations of just a few building blocks, i.e. exhibit a fast decay of coefficients. One possible way of quantifying this decay behaviour is to impose a weighted mixed ${\rm L}^p$-norm on the wavelet coefficients: Nice signals are those for which the wavelet coefficient decay is sufficiently fast to guarantee weighted integrability. One of the starting points of coorbit theory was the realization that the scale of Besov spaces can be understood in precisely these terms \cite{FeiGr1}.. 

Note however that the notion of nice signals will in general depend on the mother wavelet, which casts some doubts on the suitability of such definitions. For these reasons, a theory of nice signals needs to go hand in hand with a theory of nice wavelets, in such a way that the definition of nice signals does not depend on the choice of wavelet, as long as the latter is within the right class. For the wavelet ONB characterization of Besov spaces in dimension one \cite{DVJaPo}, there is at least one well-understood notion of nice wavelets: A wavelet ONB consisting of wavelets with suitably many vanishing moments, decay order and smoothness can be used for the characterization of Besov spaces via weighted mixed summability of the coefficients. In particular all wavelet ONB's having the prescribed properties agree on the set of nice signals. 

It is one of the main assets of coorbit theory to provide such a consistent notion of nice wavelets and signals for rather general continuous wavelet transforms: Whenever one fixes a suitable weight function and summability exponents, there exists a nonempty set $\mathcal{A}_{v_0}$ of nice wavelets such that the definition of a nice signal in terms of weighted summability of the coefficients is {\em independent} of the choice of wavelets within $\mathcal{A}_{v_0}$. Here the subscript $v_0$ serves as a reminder that the actual definition of this set will have to reflect the choice of coefficient space, just as the above-mentioned conditions for nice wavelet ONB's in the characterization of Besov spaces depend on the particular Besov space under consideration.

A second major strength of coorbit theory is that in addition to this consistency, it also allows to replace integrals by sums, and consequently to replace continuous wavelet inversion formulae by frame-type expansions converging in the proper sense, as soon as the wavelet is chosen from the (smaller) class $\mathcal{B}_{v_0}$. 

Before we give a more detailed description of the spaces $\mathcal{A}_{v_0},\mathcal{B}_{v_0}$, we need to introduce further notation. 
 A {\bf weight} on a locally compact group $K$ is a continuous function $w: K \to \mathbb{R}^+$ satisfying $w(xy) \le w(x) w(y)$, for all $x,y \in K$. 
The Besov-type coorbit spaces that we focus on in this paper are obtained by fixing a weight $v$ of the type 
\begin{equation} \label{eqn:defn_v} v(x,h) =  (1+|x| + \|h\|)^s w(h) \end{equation} on $G$, $s \ge0$, and $w$ is some weight on $H$. Note that this indeed defines a weight $v$ on the semidirect product. A weight $w$ on $H$ is called {\bf polynomially bounded} if the inequality
\[
 w(h) \preceq (1+\|h \|)^k (1+\| h^{-1}\|)^k 
\] holds, for suitably large $k$. The weights associated to the homogeneous Besov spaces, but also the examples employed in the shearlet literature (e.g. in \cite{shearlet_book}) are all polynomially bounded. 

Coorbit spaces are defined with reference to suitable solid Banach function spaces on the group $G$, see \cite{FeiGr1} for the precise requirements. In this paper, we shall concentrate on the following class of spaces: 
\begin{definition} Let $v:G \to \mathbb{R}^+$ be weight, $s\ge 0$ and $1 \le p,q < \infty$.  
We define $Y =  {\rm L}^{p,q}_v (G)$, for $1 \le p,q < \infty$ as 
\[
{\rm L}^{p,q}_v (G) = \left\{ F: G \to \mathbb{C}~:~\int_H \left(
\int_{\mathbb{R}^d} |F(x,h)|^p v(x,h)^p dx \right)^{q/p} \frac{dh}{|{\rm det}(h)|} < \infty
\right\}~,
\]
with the obvious norm, and the usual conventions regarding
identification of a.e. equal functions. We let ${\rm L}^{p}_{v} (G) = {\rm L}^{p,p}_{v}(G)$. The corresponding spaces for $p=\infty$ and/or
$q = \infty$ are defined by replacing integrals with essential
 suprema. 
 \end{definition}

 The theory of coorbit spaces largely depends on norm estimates for convolution operators, in particular applied to the reproducing kernels associated to a continuous wavelet transform. For these estimates, the next notion is indispensable: 
 \begin{definition}
  Let $Y = {\rm L}^{p,q}_v(G)$, for $1 \le p,q \le \infty$ and a weight $v$ on $G$.
   A weight $v_0$ is called {\bf control weight} for $Y$ if it satisfies 
\[
 v_0(x,h) = \Delta_G(x,h)^{-1} v_0((x,h)^{-1})~, 
\]
as well as 
 \[
\max \left( \|L_{(x,h)^{\pm 1}} \|_{Y \to Y},\| R_{(x,h)} \|_{Y \to Y},\|
R_{(x,h)^{-1}} \|_{Y \to Y} \Delta_G(x,h)^{-1} \right) \le v_0(x,h)
\] where $L_{(x,h)}$ and $R_{(x,h)}$ denote left and right translation by $(x,h) \in G$.
 \end{definition}


Given a weight $v$ as in (\ref{eqn:defn_v}), it is shown in \cite[Lemma 2.3]{Fu_coorbit} that there exists a control weight satisfying the estimate
 \begin{equation} \label{eqn:cont_weight_sep}
  v_0(x,h) \le (1+|x|)^s w_0(h)~,
 \end{equation} with $w_0 : H \to \mathbb{R}^+$ defined by 
 \begin{eqnarray*} \nonumber
  w_0(h) & = &    (w(h)+w(h^{-1})) \max \left(\Delta_G(0,h)^{-1/q}, \Delta_G(0,h)^{1/q-1} \right)  \\  & & \times \left(|{\rm det}(h)|^{1/q-1/p} + |{\rm det}(h)|^{1/p-1/q} \right) (1+\|h\|+\|h^{-1}\|)^s~.
 \end{eqnarray*} 
In particular, Lemma \ref{lem:dec_est_redundant} below will entail that $w_0$ is polynomially bounded, whenever $w$ is. 

Using $v_0$, we define
\[ \mathcal{A}_{v_0} = \{ \psi \in {\rm L}^2(\mathbb{R}^d)~:~\mathcal{W}_{\psi}\psi \in {\rm L}^1_{v_0}(G) \} \] of analyzing vectors. In order to describe the more restrictive space $\mathcal{B}_{v_0}$, we first need a further definition.
 \begin{definition}
  Let $Y$ denote a solid Banach function space on the locally compact group $G$, $U \subset G$ a compact neighborhood of the identity, and $F: G \to \mathbb{C}$. We let
  \[
   \left(\mathcal{M}_U^R F \right) (x) = \sup_{y \in U}|F(yx^{-1})|
  \] denote the right local maximum function of $F$ with respect to $U$. Given a weight $v_0$ on $G$, we denote the associated Wiener amalgam space by
  \[
   W^R(C^0,Y) = \{ F : G \to \mathbb{C}~:~ F \mbox{ continuous }, \mathcal{M}_U^R F \in Y \}~,
\]
with norm $\| F \|_{ W^R(C^0,Y)} = \|  \mathcal{M}_U^R F \|_{Y}$.

 We let
\[
 \mathcal{B}_{v_0} = \{ \psi \in {\rm L}^2(\mathbb{R}^d)~:~ \mathcal{W}_\psi \psi \in W^R(C^0,{\rm L}^1_{v_0})  \}~. 
\]
 \end{definition}
While the original sources \cite{FeiGr0,FeiGr1,FeiGr2} refer to the right Wiener amalgam space $W^R(C^0,{\rm L}^1_{v_0})$, the symmetry properties of the function $|\mathcal{W}_{\psi} \psi|$ allows to alternatively work with the left-sided  version of the amalgam norm
$\| F \|_{W(C^0,L^1_{v_0})}$, defined as the weighted $L^1_{v_0}$-norm  of the map 
 \[
  \mathcal{M}_U(F) (x) = \sup_{y \in U} |F(xy)|~.
 \] 
We will typically use this norm in the following.  Note that, since $v_0$ is bounded from below, we have $\mathcal{B}_{v_0} \subset \mathcal{A}_{v_0} \subset {\rm Co}({\rm L}^1(G)) \subset {\rm L}^2(\mathbb{R}^d)$.

Now the central results of \cite{FeiGr0,FeiGr1,FeiGr2,Gr} with respect to consistency and discretization can be summarized as follows:
\begin{enumerate}
\item Given a suitable Banach function space $Y$ on $G$ with control weight $v$, fix a nonzero $\psi \in \mathcal{A}_v$, and define 
$\mathcal{H}_{1,v} = \{ f \in {\rm L}^2(\mathbb{R}^d): \mathcal{W}_\psi f \in  {\rm L}^1_{v_0} \}$. 
Let $\mathcal{H}_{1,v}^{\sim}$ denote the space of conjugate-linear bounded functionals on $\mathcal{H}_{1,v}$. Then the wavelet transform $\mathcal{W}_\psi f$ of $f \in \mathcal{H}_{1,v}^{\sim}$ can be defined by canonical extension of the formula for ${\rm L}^2$-functions.  
\item Fix a nonzero $\psi \in \mathcal{A}_{v_0}$, and define 
\[
 {\rm Co} Y = \{ f \in \mathcal{H}_{1,v}^{\sim}~:~ \mathcal{W}_\psi f \in Y \}~, 
\] with the obvious norm $\| f \|_{{\rm Co} Y} = \| \mathcal{W}_\psi f \|_Y$. Then ${\rm Co} Y$ is a well-defined Banach space, and it is {\em independent} of the choice of $\psi \in \mathcal{A}_{v_0} \setminus \{ 0 \}$: Changing $\psi$ results in an equivalent norm. 
\item Fix a nonzero $\psi \in \mathcal{B}_{v_0}$. Then the coorbit space norm is equivalent to the discretized norm $\left\| \mathcal{W}_\psi f|_Z \right\|_{Y_d}$, for all suitably dense and discrete subsets $Z \subset G$, with a suitably defined Banach sequence space $Y_d$. This also gives rise to atomic decompositions, i.e., systems of wavelets that provide frame-like decompositions converging not just in ${\rm L}^2$, but also in the coorbit space norms. 
\end{enumerate}

\subsection{Aims of this paper}
  
Clearly, the application of coorbit theory hinges on the availability and accessibility of elements in $\mathcal{A}_{v_0}$ and $\mathcal{B}_{v_0}$. Ideally, one would wish for transparent criteria similar to the above-mentioned ones for wavelet ONB's in the one-dimensional case, i.e., in terms of smoothness, decay and vanishing moments. The papers \cite{Fu_coorbit,Fu_atom} 
are chiefly concerned with providing such criteria for general dilation groups. Here it is of key importance to employ a notion of vanishing moments that takes into account the dual orbit in a proper way. 
\begin{definition} \label{defn:van_mom}
 Let $r \in \mathbb{N}$ be given.
 $f \in {\rm L}^1(\mathbb{R}^d)$ {\bf has vanishing moments in $\mathcal{O}^c$ of order $r$}  if all
 distributional derivatives $\partial^\alpha \widehat{f}$ with $|\alpha|\le r$ are
 continuous functions, and all derivatives of degree $|\alpha|<r$ are vanishing on $\mathcal{O}^c$.
\end{definition}
Note that under suitable integrability conditions on $\psi$, the vanishing moment conditions are equivalent to 
\[
 \forall |\alpha| < r,\forall \xi \in \mathcal{O}^c ~:~\int_{\mathbb{R}^d} x^\alpha \psi(x) e^{-2 \pi i \langle \xi, x \rangle} dx = 0~. 
\]

The following theorem summarizes the chief results of \cite{Fu_coorbit,Fu_atom} regarding membership in $\mathcal{A}_{v_0}$ and $\mathcal{B}_{v_0}$. 
Informally, the significance of the theorem can be described as follows: Part (a) shows that there is a large and fairly handy class of nice wavelets around, in the form of bandlimited Schwartz functions. However, if one is interested in analogs of the above-mentioned vanishing moment criteria, parts (b) and (c) provide an answer. Finally, part (d) shows that it is easy to fulfill the criteria from (b) and (c); simply pick a reasonably nice function $f$ and apply the differential operator ${\rm D}_{\mathcal{O}}$ sufficiently often to $f$, to obtain a function fulfilling the conditions of (b) and/or (c). In particular, compactly supported atoms are obtained by starting with a suitably regular, {\em compactly supported} function $f$. 

Note however that parts (b) and (c) depend on the -- currently somewhat mysterious -- additional conditions of (strong) temperate embeddedness. We will clarify these conditions in the following section; at this point, it is sufficient if one understands them as obstacles to the applicability of the theorem. Note also the related point that in order to have concrete criteria, one needs to know the index $\ell$ occurring in each condition, or at least some estimate for it. 

\begin{theorem} \label{thm:main_cited}
Assume that the control weight $v_0$ fulfills $v_0(x,h) \le (1+|x|)^s w_0(h)$, with $s \ge 0$. 
\begin{enumerate}
\item[(a)] If $\widehat{\psi} \in C_c^\infty(\mathcal{O})$, then $\psi \in \mathcal{B}_{v_0}$. 
\item[(b)] Assume that $\mathcal{O}$ is $(s,1,w_0)$-temperately embedded with index $\ell$. Then any function $\psi$ with $|\widehat{\psi}|_{\ell+d+1,\ell+d+1} < \infty$ and vanishing moments in $\mathcal{O}^c$ of order $\ell+d+1$ is in $\mathcal{A}_{v_0}$. 
\item[(c)] Assume that $\mathcal{O}$ is strongly $(s,w_0)$-temperately embedded with index $\ell$. Then any function $\psi$ with with $|\widehat{\psi}|_{\ell+d+1,\ell+d+1} < \infty$ and vanishing moments in $\mathcal{O}^c$ of order $\ell+d+1$ is in $\mathcal{B}_{v_0}$. 
\item[(d)] There exists a partial differential operator ${\rm D}_\mathcal{O}$ with constant coefficients such that, for all $r \in \mathbb{N}$ and all functions $f$ with integrable partial derivatives of order $\le r{\rm deg}({\rm D}_{\mathcal{O}})$, the function $\psi := {\rm D}_{\mathcal{O}}^r f$ has vanishing moments in $\mathcal{O}^c$ of order $r$.  
\end{enumerate}
\end{theorem}

Part (a) is essentially \cite[Lemma 2.7]{Fu_coorbit}, parts (b) and (d) are loc. cit. Corollary 4.4 resp. Lemma 4.1, and (c) is \cite[Theorem 3.4]{Fu_atom}. 

So far, (strong) temperate embeddedness has been verified for the following dilation groups \cite{Fu_coorbit,Fu_atom}:
\begin{enumerate}
\item Diagonal groups in any dimension; 
\item similitude groups in any dimension; 
\item all possible choices of irreducibly admissible dilation groups in dimension two. 
\end{enumerate}
The different classes were checked on a case-by-case basis, with some similarities observed between the different groups, but  without a sufficiently general understanding of how larger classes of groups could be treated. 
The aims of this paper are the following:
\begin{enumerate}
\item To reduce the task of checking (strong) temperate embeddedness, which requires computing or at least estimating certain integrals and Wiener amalgam norms, to the task of comparing a handful of auxiliary functions; 
\item to demonstrate the scope of the newly derived criteria for (strong) temperate embeddedness (and consequently, the scope of Theorem \ref{thm:main_cited}), by establishing these properties for large classes of dilation groups, to wit
\begin{enumerate}
\item[i.] {\em  abelian} irreducibly admissible dilation groups,
\item[ii.] generalized shearlet dilation groups. 
\end{enumerate}
\end{enumerate}

Item (2).ii is of considerable independent interest. Following the initial construction of shearlets in dimension two, there were two distinct developments of shearlets in higher dimensions. In Section \ref{sect:gen_shearlets} we show that these groups are all special cases of a rather general construction principle, which makes the relationship between the different groups more transparent, and highlights the importance of understanding the abelian dilation groups first. It also makes a large choice of alternative shearlet dilation groups available in higher dimensions, which might be worth further exploration.

\section{Checking temperate embeddedness}

\subsection{Definition of temperate embeddedness}

 The aim of this section is the derivation of conditions on the dual action that will allow to explicitly determine sufficient vanishing moment criteria for wavelets belonging to $\mathcal{A}_{v_0}$ and $\mathcal{B}_{v_0}$. 
The central tool for this purpose is an auxiliary function $A: \mathcal{O} \to \mathbb{R}^+$ defined as follows:
Given any point $\xi \in \mathcal{O}$, let ${\rm
dist}(\xi,\mathcal{O}^c)$ denote the minimal distance of $\xi$ to
$\mathcal{O}^c$, and define
\begin{equation} \label{eqn:A_euclid}
 A(\xi) := \min \left( \frac{{\rm dist}(\xi,\mathcal{O}^c)}{1+\sqrt{|\xi|^2-{\rm dist}(\xi,\mathcal{O}^c)^2}},
 \frac{1}{1+|\xi|} \right)~. 
\end{equation}
By definition, $A$ is a continuous function with $A(\cdot) \le 1$.
If $\eta \in \mathcal{O}^c$ denotes an element of minimal distance to $\xi$, the fact that $\mathbb{R}^+ \cdot \eta  \subset \mathcal{O}^c$ then entails that $\eta$ and $\xi-\eta$ are orthogonal with respect to the standard scalar product on $\mathbb{R}^d$, and we obtain the more transparent expression
\begin{equation} \label{eqn:A_noneuclid}
  A(\xi) = \min \left( \frac{|\xi-\eta|}{1+|\eta|},
 \frac{1}{1+|\xi|} \right)~. 
\end{equation}

Using the auxiliary function $A$, we can now define the different notions of temperate embeddedness.  The first one provides access to criteria for $\mathcal{A}_{v_0}$, via Theorem \ref{thm:main_cited}(b).
\begin{definition} \label{defn:mod_embedded} Let $w: H \to \mathbb{R}^+$ denote a weight
function, $s \ge 0$, and $1 \le q < \infty$. $\mathcal{O}$ is called {\bf
$(s,q,w)$-temperately embedded (with index $\ell \in \mathbb{N}$)} if
the following two conditions hold, for a fixed $\xi_0 \in
\mathcal{O}$.
\begin{enumerate}
\item[(i)] The function $H \ni h \mapsto |{\rm det}(h)|^{1/2-1/q} (1+\|
h \|)^{s+d+1} w(h) A(h^T \xi_0)^\ell$ is in ${\rm L}^q(H)$.
\item[(ii)] The function $H \ni h \mapsto |{\rm det}(h)|^{-1/2-1/q} (1+\| h
\|)^{s+d+1} w(h) A(h^{-T} \xi_0)^\ell$ is in ${\rm L}^1(H)$.
\end{enumerate}
If $\mathcal{O}$ is $(s,q,w)$-temperately embedded for all $1 \le q <
\infty$ and $s \ge 0$, (with an index possibly depending on $s$ and $q$), the orbit
$\mathcal{O}$ is called {\bf $w$-temperately embedded}.
\end{definition}

Note that the submultiplicativity of the weights involved in the conditions easily yields that this definition is independent of the choice of $\xi_0$. 
We next turn to conditions providing vanishing moment criteria for atoms. 
For this purpose, we define a further family of auxiliary functions $\Phi_\ell : H \to \mathbb{R}^+ \cup \{ \infty \}$, for $\ell \in \mathbb{N}$, via 
\begin{equation} \label{eqn:def_Phi_ell}
 \Phi_\ell(h) =  \int_{\mathbb{R}^d} A(\xi)^\ell A(h^T \xi)^\ell d\xi
\end{equation}

Now the following definition allows to formulate sufficient vanishing moment criteria for elements of $\mathcal{B}_{w_0}$, see Theorem \ref{thm:main_cited}(c). 
\begin{definition} \label{defn:str_temperately_embed}
 Let $w: H \to \mathbb{R}^+$ denote a weight, $s \ge 0$. We call $\mathcal{O}$ {\bf strongly $(s,w)$-temperately embedded (with index $\ell \in \mathbb{N}$)} if $\Phi_\ell \in W(C^0,{\rm L}^1_{m})$, where the weight $m: H \to \mathbb{R}^+$
is defined by
\[
m(h) = w(h) |{\rm det}(h)|^{-1/2} (1+\| h \|)^{2(s+d+1)} ~.
\]
\end{definition}

\begin{remark}
 Let us give a short, informal description of the roles of the open dual orbit $\mathcal{O}$, and of the associated auxiliary functions $A$ and $\Phi_{\ell}$. The open dual orbit can be understood as the set of frequencies that the wavelet transform can easily resolve. The chief purpose of the previous papers \cite{Fu_coorbit,Fu_atom} was to establish the idea that nice wavelets are characterized by suitable space-frequency concentration, where the frequency concentration needs to be understood {\em with reference to the open dual orbit}. The first indicator that this notion is correct is provided by part (a) of Theorem \ref{thm:main_cited}: Any Schwartz function whose Fourier transform is compactly supported inside $\mathcal{O}$ is a nice wavelet. Note that, by definition, these functions have all moments vanishing in $\mathcal{O}^c$. 
 
 For compactly supported functions however, only vanishing moments of finite  order can be expected. Here one needs to quantify how a certain number of vanishing moments (together with other assumptions) translates to a suitable decay of wavelet coefficients. For this purpose the auxiliary functions $A$ and $\Phi_\ell$ are instrumental, see e.g. \cite[Lemma 3.7]{Fu_atom} for an explicit decay statement. However, in order to gauge whether the decay is in fact sufficient to conclude weighted integrability, additional requirements are necessary, whence the conditions in Definition \ref{defn:mod_embedded} arise. Finally, containment of the wavelet coefficients in a suitable Wiener amalgam space incurs further restrictions, which are reflected in Definition \ref{defn:str_temperately_embed}. 
\end{remark}

\subsection{Variations of the auxiliary function $A$} \label{subsect:norm_ind}

An often exploited feature of finite-dimensional vector spaces is that all norms on these spaces are equivalent.
The aim of this subsection is to show that the definition of the envelope functions used to establish vanishing moment criteria does not depend  in an essential way on the euclidean norm: If one defines the envelope function via an analog of equation (\ref{eqn:A_noneuclid}), with the euclidean norm replaced by any other norm, and $\eta$ denoting the distance minimizer with respect to the new norm, the resulting function is equivalent to the original envelope function; i.e., their quotient is bounded from above and away from zero. As a consequence, we will obtain a simple argument that the properties of (strong) temperate embeddedness are invariant under conjugacy. This was mentioned (and used) in \cite{Fu_coorbit,Fu_atom}, but since the norm need not be invariant under a given linear change of coordinates, we found it useful to elaborate on this point. The following somewhat technical lemma contains the main estimate. 

\begin{lemma}
Let $\mathcal{O}^c \subset \mathbb{R}^d$ denote the complement of the open dual orbit, and assume that we are given two norms $|\cdot|_i$, for $i=1,2$, on $\mathbb{R}^d$. Given $\xi \in \mathcal{O}$, we let
\[
 \eta_i = \eta_i(\xi) = \operatorname{argmin}_{y \in \mathcal{O}^c} |\xi-y|_i~;
\] in the case of more than one minimizers $\eta_i$ can be chosen arbitrarily among them. 
Furthermore, define
\[
 A_i(\xi) = \min \left( \frac{|\xi- \eta_i|_i}{1+|\eta_i|_i}, \frac{1}{1+|\xi|_i} \right)~.
\]
Then  $A_1 \preceq A_2 \preceq A_1$.
\end{lemma}
\begin{remark}
 Note that the definition of the functions $A_i$ in the lemma allows a certain ambiguity at all points $\xi$ for which the distance minimizer $\eta_i$ is not unique; here choosing a different minimizer can result in a different value of $A_i(\xi)$. It is therefore important to note that the statement holds for any choice of distance minimizers. 
\end{remark}

\begin{prf}
 By symmetry it is sufficient to prove the first inequality. 
Since all norms on finite-dimensional vector spaces are equivalent, there are constants $0 <  c_1  \le 1 \le c_2$ such that for all $x \in \mathbb{R}^d$,
\[
 c_1 |x|_1 \le |x|_2 \le c_2 |x|_1~.
\] By the choice of $\eta_1,\eta_2 \in \mathcal{O}^c$ as distance minimizers with respect to the corresponding norms, we obtain the following chain of inequalities:
\begin{equation} \label{eqn:equiv_dist}
 c_1 |\xi - \eta_1|_1 \le  c_1 |\xi - \eta_2 |_1 \le |\xi - \eta_2|_2 \le  |\xi - \eta_1|_2 \le c_2 |\xi-\eta_1|_1 ~.
\end{equation}
Furthermore, since $0 \in \mathcal{O}^c$, we also have
\begin{equation} \label{eqn:est_eta_xi}
 |\eta_i|_i \le 2 |\xi|_i~. 
\end{equation}
Here, it is important that both (\ref{eqn:equiv_dist}) and (\ref{eqn:est_eta_xi}) hold for any choice of distance minimizers $\eta_1,\eta_2$. 
Finally, we note 
 \begin{equation} \label{eqn:norm_kw}
  \frac{1}{1+|\xi|_2} \ge \frac{1}{c_2} \frac{1}{1+|\xi|_1}~. 
 \end{equation} 
We now distinguish three cases. For this purpose, fix $0 < \varepsilon < 1/2$. \\
{\bf Case 1:} $\frac{|\xi - \eta_1|_1}{|\xi|_1} < \varepsilon$.\\
In this case, the triangle inequality yields
\[
 |\eta_1|_1 \ge \frac{|\xi|_1}{2}~. 
\]
But then, using (\ref{eqn:equiv_dist}), we get 
\[ \frac{|\xi- \eta_2|_2}{1+|\eta_2|_2} \ge \frac{c_1}{2} \frac{|\xi- \eta_1|_1}{1+|\xi|_2} \ge
  \frac{c_1}{2 c_2} \frac{|\xi- \eta_1|_1}{1+|\xi|_1} \ge   \frac{c_1}{4 c_2} \frac{|\xi- \eta_1|_1}{1+|\eta_1|_1}~.
\] Combining this with (\ref{eqn:norm_kw}) yields $A_1(\xi) \preceq A_2(\xi)$. \\
 {\bf Case 2:} $\frac{|\xi - \eta_1|_1}{|\xi|_1} \ge \varepsilon$, and $\frac{|\xi-\eta_1|_1}{1+|\eta_1|_1} \le \frac{1}{1+|\xi|_1}$.\\
 Here we need to show that 
 \begin{equation} \label{eqn:casetwo}
  \frac{|\xi-\eta_1|_1}{1+|\eta_1|_1} \preceq \min \left( \frac{|\xi- \eta_2|_2}{1+|\eta_2|_2}, \frac{1}{1+|\xi|_2} \right)~.
 \end{equation}
The assumptions entail 
\[
 \frac{1}{1+|\xi|_1} \ge  \frac{|\xi-\eta_1|_1}{1+|\eta_1|_1} \ge  \frac{\varepsilon |\xi|_1}{1+|\eta_1|_1} 
\]
resulting in the quadratic inequality
\[
 1 + 2 |\xi|_1 \ge 1+|\eta_1|_1 \ge \varepsilon (1+|\xi|_1) |\xi|_1~.
\] The set of $\xi$ fulfilling this inequality is clearly bounded with respect to $|\cdot|_1$, and by norm equivalence, we obtain a $C>0$ such that
$|\xi|_2 \le C$. We then find
\[
 \frac{|\xi- \eta_2|_2}{1+|\eta_2|_2} \ge  \frac{|\xi- \eta_2|_2}{1+2C} \ge  \frac{c_1|\xi- \eta_1|_1}{1+2C} 
  \ge \frac{c_1}{1+2C} \frac{|\xi- \eta_1|_1}{1+|\eta_1|_1}~. 
\]
In combination with (\ref{eqn:norm_kw}), this entails  (\ref{eqn:casetwo}).\\
 {\bf Case 3:} $\frac{|\xi - \eta_1|_1}{|\xi|_1} \ge \varepsilon$, and $\frac{|\xi-\eta_1|_1}{1+|\eta_1|_1} > \frac{1}{1+|\xi|_1}$.\\
Here we need to prove
 \begin{equation} \label{eqn:casethree}
  \frac{1}{1+|\xi|_1} \preceq \min \left( \frac{|\xi- \eta_2|_2}{1+|\eta_2|_2}, \frac{1}{1+|\xi|_2} \right)~.
 \end{equation}
 First note that inequality (\ref{eqn:norm_kw}) immediately implies (\ref{eqn:casethree}) unless
 \[
  \frac{|\xi- \eta_2|_2}{1+|\eta_2|_2} \le \frac{1}{1+|\xi|_2}
 \] holds. Assuming this inequality, together with 
 \[
   \frac{|\xi- \eta_2|_2}{|\xi|_2} \ge \frac{c_1 |\xi-\eta_1|_1}{c_2 |\xi|_1} \ge \frac{c_1}{c_2} \varepsilon 
 \] allows to conclude, just as in Case 2, the existence of a constant $C>0$ such that $|\xi|_2 \le C$. We thus obtain
 \[
  \frac{|\xi- \eta_2|_2}{1+|\eta_2|_2}  \ge \frac{c_1}{1+2C} \frac{|\xi- \eta_1|_1}{1+|\eta_1|_1} \ge  \frac{c_1}{1+2C}   \frac{1}{1+|\xi|_1}~.
 \]
Combining this with (\ref{eqn:norm_kw}) yields  (\ref{eqn:casethree}). 
\end{prf}

\begin{lemma}
 Let $H$ be a matrix group with unique open dual orbit $\mathcal{O}_1$ and compact associated stabilizers. Let $g \in {\rm GL}(\mathbb{R}^d)$ denote an arbitrary invertible matrix, and let $H_g = g^{-1} H g$.
 Then $\mathcal{O}_2 = g^T \mathcal{O}_1$ is the unique open dual orbit of $H_g$. Let $A_i: \mathcal{O}_i \to \mathbb{R}^+$ denote the associated envelope functions, defined according to equation (\ref{eqn:A_euclid}). 
 
 Then there are constants $0 < c_1 \le c_2 $ such that for all $\xi \in \mathcal{O}_1$:  
 \[
  c_1 A_1(\xi) \le  A_2 (g^T \xi) \le c_2 A_1(\xi)~. 
 \]
In particular, $\mathcal{O}_1$ is (strongly) $(s,w_1)$-temperately embedded iff $g^T \mathcal{O}$ has the same properties, with the same index, and with reference to the pair $(s,w_2)$, where the weight $w_2$ is obtained as $w_2(h) = w_1(g h g^{-1})$. 
\end{lemma}

\begin{proof}
The statement regarding the relationship of the dual orbits follows from the calculation $(g^{-1} H g)^T (g^T \xi) = g^T H \xi$. 

 The map $\xi \mapsto A_2(g^T \xi)$ can be understood as follows: We compute $\zeta \in (g^T \mathcal{O})^c$ with minimal euclidean distance to $g^T \xi$, and obtain
 \[
  A_2(g^T \xi) = \min \left( \frac{|g^T \xi- \zeta|}{1+|\zeta|},\frac{1}{1+|g^T \xi|} \right) ~. 
 \] Here $|\cdot|$ denotes the euclidean norm. But then $\eta = g^{-T} \zeta$ can be understood as distance minimizer to $\xi$ with respect to the norm defined by $|x|_g = |g^T x|$, and we get
 \[
   A_2(g^T \xi) =  \min \left( \frac{|\xi- \eta|_g}{1+|\eta|_g},\frac{1}{1+|\xi|_g} \right)~.
 \] 
 Thus the previous lemma yields
 \[
  A_1(\xi) \preceq A_2(g^T \xi) \preceq A_1(\xi)~,
 \]
 and the remaining statements follow. 
\end{proof}

\subsection{Easily checked criteria for temperate embeddedness}

The previous papers \cite{Fu_coorbit,Fu_atom} already established (strong) temperate embeddedness for a whole class of dilation groups, including all dilation groups in dimension 2, as well as similitude and diagonal groups in arbitrary dimensions. These groups were dealt with in a case-by-case manner, each requiring somewhat different arguments and calculations. It is the aim of this section to introduce a more systematic approach to the verification of these properties. The immediate use of the following results in the course of this paper lies in their application to abelian and generalized shearlet groups later on; however, they are also of independent interest. 

The general strategy pursued in this section can be summarized as follows: Instead of studying the auxiliary function $A$ on the orbit, we study its pull-back to the dilation group via the canonical projection. This allows a unified treatment of both versions of temperate embeddedness. In addition, the auxiliary functions $\Phi_\ell$ will turn out to be weighted convolution products of the pullbacks, which will allow to use convolution inequalities for weighted amalgam space norms to establish strong temperate embeddedness. 
We therefore fix $\xi_0 \in \mathcal{O}$, and define $A_H : H \to \mathbb{R}_0^+$, $A_H(h) = A(h^T \xi_0)$. 

A further ingredient in the arguments to come is a certain Radon-Nikodym derivative. We introduce a measure $\mu_{\mathcal{O}}$ on the open orbit as the image of Haar measure under the projection map, i.e. $\mu_{\mathcal{O}}(A) = \mu_H(p_{\xi_0}^{-1}(A))$. This is a 
well-defined Radon measure on $\mathcal{O}$, and 
Lebesgue-absolutely continuous with Radon-Nikodym derivative
\[ \frac{d\mu_{\mathcal{O}}(h^T \xi_0)}{d\lambda(h^T \xi_0)} = c_0 \frac{\Delta_H(h)}{ |{\rm
det}(h)|} ~,\] for a positive constant $c_0$; see \cite{Fu96,FuDiss} for more
details.  Hence, possibly after suitable normalization, we have for all 
Borel-measurable $F: \mathcal{O} \to \mathbb{R}^+$,
\begin{equation} \label{eqn:meas_orbit_Haar}
 \int_{\mathcal{O}} F(\xi) d\xi = \int_H F(h^T \xi_0) \frac{|{\rm
 det}(h)|}{ \Delta_H(h)} dh = \int_H F(h^T \xi_0) \Delta_G(h)^{-1} dh ~.
\end{equation}

We start out with the integrability conditions in Definition \ref{defn:mod_embedded}. We note
that the conditions for the following proposition are independent of the choice of $\xi_0$, due to the
submultiplicativity of the involved weights. 
\begin{proposition} \label{prop:crit_temp_embed}
Let $w_0: H \to \mathbb{R}^+$ denote a weight on $H$, and $s\ge 0$. 
 Suppose that the auxiliary function $A_H$ fulfills the following estimates, for suitable exponents $e_1,\ldots,e_4 \ge 0$: 
 \begin{eqnarray}
   w_0(h^{\pm 1}) A_H(h)^{e_1} & \preceq & 1 \label{eqn:dec_est_A1} \\
    \| h ^{\pm 1} \| A_H(h)^{e_2} & \preceq & 1 \label{eqn:dec_est_norm} \\
 |{\rm det}(h^{\pm 1})| A_H(h)^{e_3} & \preceq &  1 \label{eqn:dec_est_det}\\
   \Delta_H(h^{\pm 1}) A_H(h)^{e_4} & \preceq & 1 ~.  \label{eqn:dec_est_A2}
 \end{eqnarray}
 Then, for all $q \ge 1$, the dual orbit $\mathcal{O}$ is $(s,q,w_0)$-temperately embedded, with index 
 \begin{equation} \label{eqn:fix_ell} \ell = \lfloor e_1 + e_2(s+d+1) + \frac{3}{2} e_3 + e_4 \rfloor +d+ 1~.\end{equation}
\end{proposition}
\begin{proof}
 We first verify integrability of 
 \[
h \mapsto   |{\rm det}(h)|^{-1/2-1/q} (1+\| h
\|)^{s+d+1} w_0(h) A(h^{-T} \xi_0)^\ell~,
 \]
 by the following calculation: We have 
 \begin{eqnarray*}
  \lefteqn{ \int_H |{\rm det}(h)|^{-1/2-1/q} (1+\| h
\|)^{s+d+1} w_0(h) A(h^{-T} \xi_0)^\ell dh} \\  & = & \int_H |{\rm det}(h)|^{1/2+1/q}
(1+\| h^{-1} \|)^{s+d+1} w_0(h^{-1}) A_H(h)^\ell  \Delta_H(h)^{-1} dh \\
& = & \int_H |{\rm det}(h)|^{1/q-1/2}
(1+\| h^{-1} \|)^{s+d+1} w_0(h^{-1}) A_H(h)^\ell  \frac{|{\rm det}(h)|}{\Delta_H(h)} dh \\
& \preceq & \int_H A_H(h)^{\ell -e_1 -e_2(s+d+1)-e_3|1/q-1/2|} \frac{|{\rm det}(h)|}{\Delta_H(h)} dh \\
& = & \int_{\mathcal{O}} A(\xi)^{\ell -e_1 -e_2(s+d+1)-e_3|1/q-1/2|} d\xi ~,
 \end{eqnarray*}
where the  inequality used assumptions (\ref{eqn:dec_est_A1}) through (\ref{eqn:dec_est_A2}), and the last equality is due to (\ref{eqn:meas_orbit_Haar}). Now $A(\xi) \le \frac{1}{1+|\xi|}$ implies finiteness of the last integral, as soon as $\ell >  d + e_1 + e_2(s+d+1)+e_3|1/q-1/2|$, which is guaranteed by (\ref{eqn:fix_ell}) and the observation that $|1/q-1/2| \le 1/2$. 

For condition (i) of temperate embeddedness, we again employ  (\ref{eqn:dec_est_A1}) through (\ref{eqn:dec_est_A2}) to obtain the estimate
\begin{eqnarray*}
   \lefteqn{ \int_H|{\rm det}(h)|^{q/2-1} (1+\|
h \|)^{q(s+d+1)} w_0(h)^q A(h^T \xi_0)^{\ell q} dh} \\
& \preceq & \int_H|{\rm det}(h)|^{q/2-2} (1+\|
h \|)^{q(s+d+1)} w_0(h)^q A(h^T \xi_0)^{\ell q-e_4}  \frac{|{\rm det}(h)|}{\Delta_H(h)}  dh
\\
& \preceq & \int_H A(h^T \xi_0)^{\ell q - q e_1 - q(s+d+1)e_2 - |q/2-2| e_3 - e_4}  \frac{|{\rm det}(h)|}{\Delta_H(h)} 
dh \\
& = & \int_{\mathcal{O}} A(\xi)^{\ell q - q e_1 - q(s+d+1)e_2 - |q/2-2| e_3 - e_4} d\xi 
\end{eqnarray*}
where the last equality was again obtained via (\ref{eqn:meas_orbit_Haar}) as above. The last integral is finite as soon as 
\[
 \ell > e_1 + e_2(s+d+1) + e_3|1/2 - 2/q| + \frac{d+e_4}{q}~.
\] This is again guaranteed by equation (\ref{eqn:fix_ell}). 
\end{proof}

Since the index $\ell$ influences the required number of vanishing moments in Theorem \ref{thm:main_cited}, one is generally interested in keeping the exponents $e_i$ as small as possible. For a quick proof of temperate embeddedness, possibly with a suboptimal $\ell$, the following lemma provides a shortcut. 
\begin{lemma} \label{lem:dec_est_redundant}
Condition (\ref{eqn:dec_est_norm}) implies (\ref{eqn:dec_est_det}) and (\ref{eqn:dec_est_A2}), with constants $e_3 = d e_2$ and $e_4 = 2e_2 {\rm dim}(H)$.
%
\end{lemma}

\begin{proof}
First note that for any linear map $T$ defined on a $k$-dimensional real vector space, the fact that the determinant is a polynomial of order $k$ implies $|{\rm det}(T)| \preceq (1+\|T\|)^k$. Thus (\ref{eqn:dec_est_norm}) implies (\ref{eqn:dec_est_det}). For the second estimate, we recall a well-known fact for Lie groups \cite[Lemma 2.30]{Folland_AHA}, namely that $\Delta_H(h) = |{\rm det}({\rm Ad}(h^{-1}))|$, with ${\rm Ad}$ denoting the adjoint action of $H$ on its Lie algebra $\mathfrak{h}$. In the current setting, where $H$ is a closed matrix group, we can identify $\mathfrak{h}$ with a matrix Lie algebra, and the adjoint action is then given by ${\rm Ad}(h)(X) = h X h^{-1}$. In particular, we obtain $\|{\rm Ad}(h) \| \le \| h \| \| h^{-1}\|$, and consequently the above observation yields
\[
\Delta_H(h)  \preceq \left(1+\| h \| ~\| h^{-1}\|\right)^{{\rm dim}(H)} ,
\] whence we obtain the estimate for $e_4$. 
%
%
%
\end{proof}

Note that if we assume in addition that $w_0$ is a polynomially bounded weight, the lemma implies that temperate embeddedness can be deduced from inequality (\ref{eqn:dec_est_norm}) alone. 

A further simple but useful observation concerns direct products.
\begin{lemma} \label{lem:prod_groups}
Assume that there exist matrix groups $H_1,H_2$ such that 
\[
H = \left\{ \left( \begin{array}{cc} h_1 & 0 \\ 0 & h_2 \end{array} \right) : h_i \in H_i \right\}.  
\] Let $w_0$ be a weight on $H$, and denote by $w_i$ its restriction to $H_i$ (canonically identified with a subgroup of $H$), then we have $w_0(h_1,h_2) \le w_1(h_1) w_2(h_2)$. If the $H_i$ fulfill the estimates (\ref{eqn:dec_est_A1})-(\ref{eqn:dec_est_A2}) with exponents $e_{1,i}, \ldots,e_{4,i}$, then $H$ fulfills the same estimates, with exponents 
\[ e_1 = e_{1,1} + e_{1,2} ~,~ e_2= \max (e_{2,1},e_{2,2})~,~ e_3 = e_{3,1}+e_{3,2}~,~ e_4 = e_{4,1}+e_{4,2}~.\]
\end{lemma}
\begin{proof}
This follows from the submultiplicativity properties of the involved quantities (note the exception for the norm, due to the block diagonal structure), together with the estimate $A_H(h)^2 \le A_{H_1} (h_1) A_{H_2}(h_2)$, see the proof of \cite[Lemma 4.6]{Fu_atom}. 
\end{proof}

As strong temperate embeddedness involves checking the Wiener amalgam norms of the functions $\Phi_\ell$, it is typically harder to verify. We will nonetheless show that, here as well, the estimates (\ref{eqn:dec_est_A1})-(\ref{eqn:dec_est_A2}) are sufficient.  Again, note that this further simplifies to just verifying (\ref{eqn:dec_est_norm}) whenever $w_0$ is polynomially bounded. We perform the required estimate of the Wiener amalgam norm using a two-step procedure: Step one exhibits the function $\Phi_\ell$ as a convolution product. This will allow to reduce the problem to that of verifying whether a suitable power of $A_H$ is contained in a certain Wiener amalgam space. 

\begin{lemma}
 \label{lem:Phi_l_conv}
 \begin{enumerate}
  \item[(a)] The auxiliary functions $\Phi_\ell$ and $A_H$ are related by 
  \begin{equation}
  \Phi_\ell =  (A_H^\ell |{\rm det}(\cdot)|)^{\sim}  \ast A_H^\ell ~.
 \end{equation}
 Here we used the notation $F^\sim(h) = F(h^{-1})$, for any function $F$ on $H$. 
 \item[(b)] Let $m$ be any weight on $H$. Then $\Phi_\ell$ is contained in $W(C^0,{\rm L}^1_{m})$ whenever \[ A_H^\ell \in {\rm L}^1_{(\Delta_G m)^\sim}(H) \cap W(C^0,{\rm L}^1_{m})  ~.\]
 \end{enumerate}
\end{lemma}
\begin{proof}
 Using relation  (\ref{eqn:meas_orbit_Haar}), we obtain
 \begin{eqnarray*}
  \Phi_\ell(h) & = & \int_{\mathcal{O}} A(h^T \xi)^\ell A(\xi)^\ell d\xi \\
  & = & \int_H A(h^T g^T \xi_0)^\ell A(g^T \xi_0)^\ell  \Delta_G(g)^{-1} dg  \\
  & = & \int_H A_H(gh)^\ell A_H (g)^\ell |{\rm det}(g)| \Delta_H(g)^{-1} dg \\
  & = & \int_H A_H(g^{-1}h)^\ell   A_H(g^{-1})^\ell |{\rm det}(g^{-1})| dg \\
  & = & \int_H (A_H ^\ell |{\rm det}(\cdot)|)^{\sim}(g)  A_H(g^{-1} h)^\ell dg \\
  & = &   (A_H^\ell |{\rm det}(\cdot)|)^{\sim}  \ast A_H^\ell (h)~,
 \end{eqnarray*}
which proves (a). 

Part (b) follows from (a) by employing generalizations of Young's Theorem for Wiener amalgam spaces, more specifically, the estimate
\begin{equation}
 \label{eqn:Young_amalgam}
 \| f \ast g \|_{W(C^0,{\rm L}^1_m)} \le \| f \|_{{\rm L}^1_m} \| g \|_{W(C^0,{\rm L}^1_m)}~, 
\end{equation}
valid for any submultiplicative weight function $m$ and continuous $f,g: H \to \mathbb{C}$.
Estimates of this type are at the core of coorbit theory, and can be found in the original sources \cite{FeiGr0,FeiGr1,FeiGr2}. We include a proof for the sake of completeness. We first observe that
\begin{eqnarray*}
 \left(\mathcal{M}_U (f \ast g) \right) (x) & = & \sup_{z \in U}\left| \int_H f(y) g(y^{-1}xz) dx \right| \\
  & \le & \int_H \sup_{z \in U} |f(y) g(y^{-1} x z)| dx \\
  & = & \left( |f| \ast \mathcal{M}_U g \right) (x)~.
  \end{eqnarray*}
  As a consequence, 
  \begin{eqnarray*}
   \| f \ast g \|_{W(C^0,{\rm L}^1_m)}  & = & \int_H \left(\mathcal{M}_U (f \ast g) \right) (x) m(x) dx \\
   & \le & \int_H \int_H |f(y)| \mathcal{M}_U g(y^{-1}x) m(x) dy dx \\
   & \le & \int_H \int_H |f(y)| m(y)  \left( \mathcal{M}_U g \right) (y^{-1}x) m(y^{-1}x) dx dy \\
   & = & \left\| (|f| m) \ast \left(   \left( \mathcal{M}_U g \right) m \right) \right\|_1 \\
   & \le & \| f \|_{{\rm L}^1_m} \| g \|_{W(C^0, {\rm L}^1_m)}~,
  \end{eqnarray*}
 where the last inequality is due to Young's inequality for ${\rm L}^1$. 
 Applying this to $f = A_H^\ell |{\rm det}(\cdot)|, g = A_H^\ell$, and observing that
 \begin{eqnarray*}
  \left\| \left(A_H^\ell  |{\rm det}(\cdot)| \right)^\sim \right\|_{{\rm L}^1_m} & = & \int_H A_H^{\ell}(h^{-1}) |{\rm det}(h^{-1})| m(h) dh \\
  & = & \int_H A_H^\ell(h) m(h^{-1})|{\rm det}(h)| \Delta_H(h)^{-1} dh ~,
 \end{eqnarray*}
we obtain (b). 
\end{proof}

The following technical lemma provides the key to step two, the Wiener amalgam norm estimates for $A_H^\ell$:
\begin{lemma}
 \label{lem:A_mod} There exists a neighborhood $U \subset H$ of the identity and $C>0$ such that
 \begin{equation} \label{eqn:A_mod}
\forall \xi \in \mathcal{O}~ \forall h \in U ~:~ A(h^T \xi) \le C A (\xi)~. 
 \end{equation}
As a consequence, for any solid Banach function space $Y$ on $H$: 
\begin{equation} \label{eqn:A_H_amalgam} A_H^\ell \in W(C^0,Y) \Leftrightarrow A_H^\ell \in Y~.  
\end{equation}
\end{lemma}

\begin{proof}
The equivalence (\ref{eqn:A_H_amalgam}) follows from (\ref{eqn:A_mod}), since the latter implies
\[
A_H^\ell(h) \le \mathcal{M}_U(A_H^\ell)(h) \le C^\ell A_H^\ell(h)~. 
\]

For the proof of (\ref{eqn:A_mod}) we let $U = \{ h \in H : \| h - {\rm id} \| < 1/2 \}$. Then the triangle inequality yields
 \[
 \frac{1}{2} | \xi | \le  |h^T \xi | \le \frac32 | \xi |~,
 \] for all $\xi \in \mathbb{R}^n$. Now let $\xi \in \mathcal{O}$ and $h \in U$. We let $\xi' \in \mathcal{O}^c$ denote an element of minimal euclidean distance to $\xi$, and $\xi'' \in \mathcal{O}^c$ an element of minimal euclidean distance to $h^T \xi$. Recall from Subsection \ref{subsect:norm_ind} that we have
 \[
  A(\xi) = \min\left( \frac{|\xi - \xi'|}{1+|\xi'|},\frac{1}{1+|\xi|} \right)~,~A(h^T \xi)  = \min\left( \frac{|h^T \xi - \xi''|}{1+|\xi''|}, \frac{1}{1+|h^T \xi|} \right)~.
 \]
 Since $|h^T \xi| \ge \frac{1}{2} |\xi|$, the second terms involved in determining $A(\xi)$ and $A(h^T \xi)$ fulfill 
\begin{equation} \label{eqn:A_H_mod_part}
  \frac{1}{1+|h^T \xi|}  \le 2  \frac{1}{1+|\xi|}
\end{equation}
 Furthermore, invariance of $\mathcal{O}$ under $H^T$ implies the same for the complement, in particular $h^T \xi' \in \mathcal{O}^c$. Hence by choice of $\xi''$ as distance minimizer to $h^T \xi$: 
 \[
  | h^T \xi - \xi'' | \le |h^T (\xi - \xi')| \le \frac32 |\xi-\xi'| ~. 
 \]

 For the proof of (\ref{eqn:A_mod}), we now distinguish four cases.\\
 {\bf Case 1: $16 |\xi - \xi'|^2 \le |\xi'|^2$.}\\
 Here we have 
 \begin{eqnarray*}
  |\xi''|^2 & = & |h^T \xi|^2 - |h^T \xi - \xi''|^2 \ge |h^T \xi|^2 - |h^T(\xi- \xi')|^2 \\
  & \ge & \frac14 |\xi|^2 - \frac{9}{4} |\xi-\xi'|^2 \\
  & = & \frac14 |\xi'|^2 - 2 |\xi -\xi'|^2 \ge \frac18 |\xi'|^2~. 
 \end{eqnarray*}
But then we get
\[
 \frac{|h^T \xi - \xi''|}{1+|\xi''|} \le \frac{\frac32 |\xi - \xi'|}{1+ \frac{1}{\sqrt{8}} |\xi'|} \le 3 \sqrt{2} \frac{|\xi - \xi'|}{1+|\xi'|}~. 
\]
Combined with equation (\ref{eqn:A_H_mod_part}), this yields $A(h^T \xi) \le 3 \sqrt{2} A(\xi)$. \\
{\bf Case 2: $16 |\xi - \xi'|^2 > |\xi'|^2, |\xi'| > 1$.}\\
In this case we find 
\[
 \frac{|\xi-\xi'|}{1+|\xi'|} \ge \frac{1}{4} \frac{|\xi'|}{1+|\xi'|} > \frac{1}{4} \frac{1}{1+|\xi|}~,
\] and thus 
\[
 A(\xi) \ge  \frac{1}{4} \frac{1}{1+|\xi|}~.
\] But this implies
\[
 A(h^T \xi) \le \frac{1}{1+|h^T \xi|} \le 2 \frac{1}{1+|\xi|} \le 8 A(\xi)~. 
\]
{\bf Case 3: $1 \ge 16 |\xi - \xi'|^2 > |\xi'|^2, |\xi'| \le 1$.}\\
In this setting we have 
\[
  |\xi|^2 = |\xi-\xi'|^2 + |\xi'|^2 \le 17 |\xi -\xi'|^2~;
\]
in particular
\[
 |\xi| \le 5 | \xi - \xi'| \le 5~. 
\]
On the one hand, this implies that 
\[
\frac{|\xi|}{30} \le  \frac{|\xi|}{10} \le \frac{|\xi-\xi'|}{2} \le \frac{|\xi-\xi'|}{1+|\xi'|}~.
\]
On the other hand, we get
\[
 \frac{1}{1+|\xi|} \ge \frac{1}{6} \ge \frac{|\xi|}{30}~,
\]
which finally yields 
\[
 A(\xi) \ge \frac{|\xi|}{30}~.
\]  But then
\[
 A(h^T \xi) \le |h^T \xi - \xi''| \le |h^T \xi| \le \frac{3}{2} |\xi| \le 45 A(\xi)~.  
\]
{\bf Case 4: $ 16 |\xi - \xi'|^2 > \max(|\xi'|^2,1), |\xi'| \le 1$.}\\
Here the fact that $|\xi'| \le |\xi|$ yields 
\[
 \frac{|\xi-\xi'|}{1+|\xi'|} \ge \frac{1}{4} \frac{1}{1+|\xi|}~,
\] and thus $A(\xi) \ge \frac{1}{4} \frac{1}{1+|\xi|}$. Now the same reasoning as in Case 2 results in
$A(h^T \xi) \le 8 A(\xi)$.
\end{proof}

\begin{theorem}\label{thm:crit_str_temp_embed}
Let $w_0: H \to \mathbb{R}^+$ denote a weight on $H$, and $s\ge0$. 
 Suppose that the auxiliary function $A_H$ fulfills the estimates (\ref{eqn:dec_est_A1}) through (\ref{eqn:dec_est_A2}), for some exponents $e_1,\ldots,e_4$.
 Then $\mathcal{O}$ is strongly $(s,w_0)$-temperately embedded, with index 
 \[ \ell = \lfloor e_1 + e_2(2s+2d+2) + \frac{3}{2} e_3 +  e_4 \rfloor + d+1 ~.\]
\end{theorem}
\begin{proof}
 By the previous two lemmas we need to show that 
 \[ A_H^\ell \in {\rm L}^1_{(\Delta_G m)^\sim}(H)\cap {\rm L}^1_{m}(H)  ~,\]
 where $m(h) = w_0(h) |{\rm det}(h)|^{-1/2} (1+\| h \|)^{2(s+d+1)}$.
 This is done in the same way as in the proof of Proposition \ref{prop:crit_temp_embed}: 
 \begin{eqnarray*}
  \| A_H^\ell \|_{{\rm L}^1_m} & = & \int_H A_H^\ell(h) w_0(h) |{\rm det}(h)|^{-1/2} (1+\| h \|)^{2(s+d+1)} dh \\
   & \preceq & \int_H A_H(h)^{\ell - e_1 -2 e_2(s+d+1) - e_3/2 } dh \\
    & \preceq & \int_H A_H(h)^{\ell - e_1 -2 e_2(s+d+1) - e_3/2 -e_3 - e_4} \Delta_G(h)^{-1} dh \\
    & = & \int_{\mathcal{O}} A(\xi)^{\ell  - e_1 -2 e_2(s+d+1) - \frac{3}{2} e_3 - e_4} d\xi \\
    & < & \infty~,
 \end{eqnarray*}
 by our choice of $\ell$. 
Similarly, 
\begin{eqnarray*}
  \| A_H^\ell \|_{{\rm L}^1_{(\Delta_G m)^\sim}} & = & 
   \int_H A_H^\ell(h) w_0(h^{-1}) |{\rm det}(h)|^{1/2} (1+\| h^{-1} \|)^{2(s+d+1)} \Delta_G(h)^{-1} dh \\
    & \preceq & \int_H A_H(h)^{\ell -  e_1 -2 e_2(s+d+1)-e3/2} \Delta_G(h)^{-1} dh \\
    & = & \int_{\mathcal{O}} A(\xi)^{\ell -  e_1 -2 e_2(s+d+1)-e_3/2} d\xi \\
    & < & \infty~.
\end{eqnarray*}
\end{proof}

\begin{remark} \label{rem:dim_two}
 By comparison to the quite extensive computations in the previous papers  \cite{Fu_coorbit, Fu_atom}, the verification of  (strong) temperate embeddedness for all two-dimensional irreducibly admissible groups via \ref{prop:crit_temp_embed} and  \ref{thm:crit_str_temp_embed} is almost effortless. Recall that we have seen in Subsection \ref{subsect:norm_ind} that (strong) temperate embeddedness is inherited by conjugate groups. Hence, by the classification results in \cite{FuDiss}, the following list covers all irreducibly admissible matrix groups in dimension two. 
 
 In the following, we consider the coorbit spaces ${\rm Co} ({\rm L}^p(G))$ for simplicity, with $p \in [1,\infty]$. By \cite[Lemma 2.3]{Fu_coorbit} and the subsequent remark, we can take $v_0(x,h) = w_0(h) = {\rm max}(1,\Delta_G(h))$ as control weight. Adaptations to more general coefficient spaces such as mixed weighted $L^p$-spaces are straightforward, at least for polynomially bounded weights. Typically, the only additional work to be done is to compute $e_1$ and employ the proper value for $s$, in case the weight depends on the translation variable. 
 \begin{itemize}
  \item {\bf Case 1}: The {\em similitude group} $H = \mathbb{R}^+ \cdot SO(d)$, introduced by Murenzi \cite{Mu}. Elements of $H$ are written as $h = r S$, with $r>0$ and $S$ a rotation matrix, and the open dual orbit is $\mathcal{O} = \mathbb{R}^d \setminus \{ 0 \}$. Hence for the associated differential operator ${\rm D}_{\mathcal{O}}$, one can take the Laplacian, which needs to be applied $\left\lceil \frac{t}{2} \right\rceil$ times to induce vanishing moments of order $t$.  For the auxiliary function, we obtain 
  \[
   A_H(h) = \min\left( |r|,\frac{1}{1+|r|} \right)~,~|{\rm det}(h)| = r^d~,\Delta_H (h) = 1~,
  \] 
  and thus 
  \[
   e_1 = d~,~e_2 = 1~,~e_3 = d~,~e_4 = 0~.
  \] This yields temperate embeddedness with index
  \[
  \ell_1 = \left\lfloor \frac{d}{2} \right\rfloor + 4d + 1
  \] and strong temperate embeddedness with index
  \[
   \ell_2 = \left\lfloor \frac{d}{2} \right\rfloor + 5d +2~,
  \] jointly for all $p$; and adding $d+1$ to these quantities yields sufficient numbers of vanishing moments for analyzing vectors and atoms. 
  \item {\bf Case 2}: The $d$-dimensional {\em diagonal group}. For $d=1$, we obtain $e_1 = e_2 = e_3 = 1, e_4=0$ from the case just considered. Using Lemma \ref{lem:prod_groups}, we obtain for the general case
  \[
   e_1 = d~,~e_2 = 1~,~e_3 = d~,~ e_4 = 0~,
  \] i.e. the {\em same} indices as in the similitude group case, yielding temperate embeddedness with 
   \[
  \ell_1 = \left\lfloor \frac{d}{2} \right\rfloor + 4d + 1
  \] and strong temperate embeddedness with index
  \[
   \ell_2 = \left\lfloor \frac{d}{2} \right\rfloor + 5d +2~,
  \] jointly for all $p$.
  
  Here, the open dual orbit is 
  \[
   \mathcal{O} = \left\{ \xi = (\xi_1,\ldots,\xi_d)^T \in \mathbb{R}^d~:~\prod_{i=1}^d \xi_d \not= 0 \right\}~, 
  \]
and the associated differential operator ${\rm D}_{\mathcal{O}}$ can be chosen as
\[
 {\rm D}_{\mathcal{O}} f  = \frac{{\rm d}}{{\rm d} x_1} \frac{{\rm d}}{{\rm d} x_2} \ldots \frac{{\rm d}}{{\rm d} x_d} f~. 
\]
  \item {\bf Case 3}: The shearlet-type groups, given by 
\[
H =  H_c = \left\{ \pm \left( \begin{array}{cc} a & b \\ 0 & a^c \end{array}
 \right)~:~a, b \in \mathbb{R}, a > 0 \right\}~.
\] Here $c$ can be any real number.

 Haar measure on $H$  is given by $db \frac{da}{|a|^{2}}$, the modular function is $\Delta_H(h) = |a|^{c-1}$, resulting in $\Delta_G(h) = |a|^{-2}$. The dual orbit is
computed as
\[ \mathcal{O} =
 \mathbb{R}^2 \setminus (\{ 0 \} \times \mathbb{R}))~.\]
Here, the natural choice for the associated differential operator is ${\rm D}_{\mathcal{O}} = \frac{{\rm d}}{{\rm d} x_1}$. 
For $h = \left( \begin{array}{cc} a & b \\ 0 & a^c
\end{array}
 \right) \in H$ and $\xi_0 = (1,0)^T \in \mathcal{O}$, we obtain $h^T \xi_0 = (a,b)^T$, and thus
 \[
 A_H(h) = \min \left( \frac{|a|}{1+|b|},\frac{1}{1+|(a,b)^T|}\right)~. 
 \]
 In particular, we find 
 \[
  \| h \| A_H(h)^{\max(1,|c|)} \preceq 1~, 
 \] as well as
 \begin{eqnarray*}
  \| h^{-1} \| A_H(h)^{1+|c|} & = & \left\|  \left( \begin{array}{cc} a^{-1} & -a^{-1-c}b \\ 0 & a^{-c}
\end{array}
 \right) \right\| A_H(h)^{1+|c|} \\
 & \asymp & \max\left( |a|^{-1}, |a|^{-c},|a|^{-1-c}|b| \right) A_H(h)^{1+|c|} \preceq 1
 \end{eqnarray*}
yielding $e_2 = 1+|c|$. The remaining constants can be taken as 
\[
e_1 = 2~,~ e_3 = |1+c|~,~e_4 = |1-c|~,
\]
 and we obtain temperate embeddedness with index 
\[
 \ell_1 = \left\lfloor 3 |c| + |1+c| \frac{3}{2} + |1-c| \right\rfloor +8~,
\]
as well as strong temperate embeddedness with index
\[
 \ell_2 = \left\lfloor  6|c| + |1+c| \frac{3}{2} + |1-c| \right\rfloor +11~. 
\]
Thus, for the case $c=1/2$ corresponding to the original shearlet group, Theorem \ref{thm:main_cited} (b) requires vanishing moments of order 15, whereas part (c) requires order 19. This is quite possibly a rather conservative estimate, but it compares favourably with the only explicit previous estimate for part (c), which is the value $127$ obtained in \cite{Fu_atom}. 
 \end{itemize}

\end{remark}

\section{Abelian dilation groups}

In this section, we assume that $H< {\rm GL}(\mathbb{R}^d)$ is an {\em abelian} irreducibly admissible matrix group. These groups are well-understood, thanks to their close relationship to associative commutative algebras of the same dimension, observed in \cite{Fu_abelian}. Before we state the main results from that paper, it is appropriate to recall a few basic facts from the Lie theory of matrix groups, beginning with the definition of the matrix exponential: 
\[
 \exp(Y) = \sum_{k=0}^\infty \frac{Y^k}{k!}~,
\] which converges for all matrices $Y \in \mathbb{R}^{d \times d}$. We will write $T(\mathbb{R},d) \subset {\rm GL}(\mathbb{R},d)$ for the subgroup of unipotent matrices, i.e, the subgroup consisting of upper triangular matrices with ones on the diagonal. $\mathfrak{t}$ denotes the associated Lie algebra, which consists of all properly upper triangular matrices. Furthermore, we let $\mathfrak{gl}(n,\mathbb{R}) = \mathbb{R}^{d \times d}$, and identify its elements with the linear endomorphisms of $\mathbb{R}^d$ induced by matrix multiplication. 
Given a closed matrix group $H$, the Lie algebra $\mathfrak{h}$ of $H$ is the set of all matrices $Y$ satisfying $\exp(\mathbb{R}Y) \subset H$. It is a Lie subalgebra $\mathfrak{gl}(n,\mathbb{R})$, i.e., for all $X,Y \in \mathfrak{h}$, we have 
\[
 [X,Y] = XY - YX \in \mathfrak{h}~. 
\] 

The following theorem summarizes Theorems 11, 13 and Proposition 12 of \cite{Fu_abelian}. For its formulation, we use $\mathcal{A}^\times$ for the group of multiplicatively invertible elements contained in an algebra $\mathcal{A}$ with unity. 
\begin{theorem} \label{thm:ab_dil}
Let $H< {\rm GL}(\mathbb{R}^d)$ be a closed abelian matrix group. Then the following are equivalent:
\begin{enumerate}
\item[(a)] $H$ is irreducibly admissible. 
\item[(b)] The space $\mathfrak{h} = {\rm span}(H)$ spanned by $H$ is an associative subalgebra of $\mathfrak{gl}(d,\mathbb{R}) = \mathbb{R}^{d \times d}$ containing the identity, and $H= \mathfrak{h}^\times$. In addition, $\mathfrak{h}$ is the Lie algebra of $H$, and there exists $\xi \in \mathbb{R}^d$ such that $\mathfrak{h} \ni X \mapsto X^T \xi \in \mathbb{R}^d$ is a bijection.
\item[(c)] There exists a $d$-dimensional commutative algebra $\mathcal{A}$ with unity and a linear isomorphism $\psi: \mathcal{A} \to \mathbb{R}^d$ such that 
\[ H = \{ h : \exists a \in \mathcal{A}^\times \mbox{ such that } \forall \xi \in \mathbb{R}^d,  h^T \xi = \psi(a \psi^{-1} (\xi)) \}. \]
\end{enumerate}
Furthermore, two closed abelian matrix groups satisfying (a) - (c) are conjugate inside ${\rm GL}(d,\mathbb{R})$ iff the commutative algebras associated via (c) are isomorphic.
\end{theorem}
Thus the theorem sets up a bijection between isomorphism classes of abelian associative algebras with unity and conjugacy classes of irreducibly admissible abelian dilation groups. The linear map
\[
 \rho: \mathcal{A} \ni a \mapsto \left( \xi \mapsto \psi(a \psi^{-1} (\xi) ) \right)
\] is called the {\bf regular representation} of the algebra $\mathcal{A}$.

A further simplification can be obtained by decomposing the algebra into irreducible components. We call an algebra $\mathcal{A}$ over the field $\mathbb{K} \in \{ \mathbb{R},\mathbb{C} \}$ {\bf irreducible}, if $\mathcal{A} = 1_{\mathcal{A}} \cdot \mathbb{K} + \mathcal{N}$, with 
\[ \mathcal{N} = \{ a \in \mathcal{A} : a^k = 0 \mbox{ for some } k>0 \}, \] denoting the {\bf nilradical} of $\mathcal{A}$, i.e. the subalgebra of nilpotent elements. Now any commutative algebra with unity is a unique direct sum of irreducible subalgebras \cite[Theorem 17]{Fu_abelian}.

As a consequence, after a suitable change of coordinates any irreducibly admissible abelian dilation group can be written as a direct product 
\begin{equation} \label{eqn:block_diagonal} H = \left\{ h = \left( \begin{array}{cccc} h_1 & & & \\ & h_2 & & \\ & & \ddots & \\ & & & h_k \end{array} \right) : h_i \in H_i \right\}. \end{equation}
Here each $H_i$ is an irreducibly admissible abelian dilation group in (real) dimension $d_i$, and is in addition the unit group of an irreducible algebra. However, one quickly realizes that the invertible elements in an irreducible algebra are precisely those of the form $r \cdot 1_{\mathcal{A}_i} + a$, with $r \not= 0$, and $a \in \mathcal{N}$. Thus 
\[ H_i \cong \mathbb{K}^\times \cdot 1_{\mathcal{A}_i} + \mathcal{N}_i .\]
With all these observations in place, it is not difficult to establish the following result, which makes the conclusions of Theorem \ref{thm:main_cited} available for all abelian irreducibly admissible dilation groups. 
For a sharper estimate of the constants, we need one more piece of terminology: Given an irreducible algebra $\mathcal{A} = 1_{\mathcal{A}} \cdot \mathbb{K} + \mathcal{N}$, we let $n(\mathcal{A})$ denotes the {\bf nilpotency class} of $\mathcal{A}$, which is the minimal exponent $n$ such that $\mathcal{N}^n = \{ 0 \}$.

\begin{theorem} \label{thm:abelian_temp_embed}
 Let $H$ denote an irreducibly admissible abelian dilation group. 
 Let $n_1,\ldots, n_k$ denote the nilpotency classes of the irreducible blocks entering in the decomposition (\ref{eqn:block_diagonal}) of $H$, and let $n = \max_{i=1,\ldots,k} n_i$. 
 
 Then $H$ fulfills the estimates (\ref{eqn:dec_est_norm})-(\ref{eqn:dec_est_A2}), with exponents
 \[
  e_2 = 2n-1~,~ e_3 = d~,~ e_4 = 0~.  
 \]
In particular, for given polynomial weight $w_0$ on $H$,  $q \in [1,\infty)$ and $s\ge0$, the open dual orbit is both $(s,q,w_0)$-temperately embedded and strongly $(s,w_0)$-temperately embedded, for suitable choices of indices.  
\end{theorem}
\begin{proof}
In view of (\ref{eqn:block_diagonal}) and Lemma \ref{lem:prod_groups}, we directly consider the case that the algebra $\mathcal{A}$ associated to $H$ is irreducible, i.e., 
\[
 \mathcal{A} = \mathbb{K} \cdot 1_\mathcal{A} + \mathcal{N}
\] with a suitable nilpotent algebra $\mathcal{N}$ over $\mathbb{K} \in \{ \mathbb{R}, \mathbb{C} \}$,
and with nilpotency class $n$. 
We endow $\mathcal{A}$ with any euclidean norm $|\cdot |_{\mathcal{A}}$ such that the direct sum $\mathbb{K} \cdot 1_{\mathcal{A}} + \mathcal{N}$ is orthogonal. We fix an isometric linear bijection $\psi: \mathcal{A} \to \mathbb{R}^d$. Then, by part (c) of Theorem \ref{thm:ab_dil}, $H^T$ is conjugate to the image of $\mathcal{A}$ under the regular representation $\rho$, and for $\xi = \psi(1_{\mathcal{A}})$ and $h = \rho(a)^T$, we find that $h^T \xi = \psi(a)$. 

These considerations show that we are allowed to work with the following identifications: $H = \mathbb{K}^\times + \mathcal{N}$, and for $h = r \cdot 1_{\mathcal{A}} + a \in H$, with $a \in \mathcal{N}$ and $r \in \mathbb{K}^\times$, we have
\[
 A_H(h) = \min \left( \frac{|r|}{1+|a|_{\mathcal{A}}}, \frac{1}{1+\sqrt{|r|^2+ |a|_{\mathcal{A}}^2}} \right)~.
\] Furthermore, the operator norm on the left-hand side of (\ref{eqn:dec_est_norm}) is equivalent to $|h|_{\mathcal{A}}
= \sqrt{|r|^2+ |a|^2_{\mathcal{A}}}$. Thus we immediately see that 
\[
 \| h \| A_H(h) \preceq 1~. 
\] Since $a^n=0$, the inverse of $h$ can be computed by a truncated Neumann series,
\[
 h^{-1} = r^{-1}\left(\sum_{j=0}^{n-1} (-1)^j r^{-j} a^j \right)~. 
\] Since the bilinear multiplication map $\mathcal{A} \times \mathcal{A} \to \mathcal{A}$ is bounded with respect to any choice of norms, we obtain, for a suitable constant $C>0$ and all $a,b \in \mathcal{A} : |ab|_{\mathcal{A}} \le C |a|_{\mathcal{A}} |b|_{\mathcal{A}}$. But this implies
\[
 \|h^{-1} \| \preceq |r|^{-1} \sum_{j=0}^{n-1} |r|^{-j} |a|_{\mathcal{A}}^{j} \le  |r|^{-1} (1+|a|)^{n-1} (1+|r|^{-1})^{n-1}~,
\]
which finally yields
\[
 \| h^{-1} \| A_H(h)^{2n-1} \preceq 1~.
\]
Thus we obtain the estimate (\ref{eqn:dec_est_norm}) with $e_2 = 2n-1$. Instead of invoking Lemma \ref{lem:dec_est_redundant} for the remaining inequalities, we may as well observe directly that being abelian, $H$ is  unimodular, and thus (\ref{eqn:dec_est_A2}) holds with $e_4=0$. Furthermore, for $h =  r \cdot 1_{\mathcal{A}} + a$, the fact that ${\rm det}({\rm id} + a) = 1$ for nilpotent matrices $a$ yields $|{\rm det}(h)| = |r|^{d}$, thus (\ref{eqn:dec_est_det}) holds with $e_3 = d$.
\end{proof}

\section{Generalized shearlet dilation groups}

\label{sect:gen_shearlets}

In this section, we consider a rather general class of shearlet groups, extending the examples studied in  \cite{DaStTe10,DaStTe12,CzaKi12,DaTe10,DaHaTe12}. Our approach provides a common framework for the treatment of all these groups, by establishing a previously unobserved connection to commutative associative algebras. It turns out that the additional algebraic structure is quite useful for the study of admissibility and vanishing moment conditions. 

\begin{definition}
Let $H < {\rm GL}(d,\mathbb{R})$ denote an irreducibly admissible dilation group. $H$ is called {\bf generalized shearlet dilation group}, if there exist two closed subgroups $S, D < H$ with the following properties:
\begin{enumerate}
\item[(i)] $S$ is a connected closed abelian subgroup of $T(d,\mathbb{R})$. 
\item[(ii)] $D = \{ \exp(r Y) : r \in \mathbb{R} \}$ is a one-parameter group, where $Y$ is a diagonal matrix. 
\item[(iii)] Every $h \in H$ can be written uniquely as $h = \pm d s$, with $d \in D$ and $s \in S$. 
\end{enumerate}
$S$ is called the {\bf shearing subgroup} of $H$, and $D$ is called the {\bf diagonal complement} of $S$.
\end{definition}
 
\begin{remark} \label{rem:ex_shearlet}
With this definition, one quickly realizes that in dimension two, there is only one possible candidate of a shearing subgroup. For higher dimensions, two distinct candidates for shearing subgroups have been considered so far:
\begin{enumerate}
\item[(i)] The following shearing group was studied, e.g., in \cite{DaStTe10,DaStTe12,CzaKi12,dahlke2014different}:
\[ S = \left\{ \left( \begin{array}{cccc} 1 & s_1 &\ldots & s_{d-1} \\ 
  & 1 &   &  \\  &  & \ddots & \\ &  & & 1 \end{array} \right) : s_1,\ldots,s_{d-1} \in \mathbb{R} \right\}~~.\]
\item[(ii)] The {\em Toeplitz shearing group} was proposed by Dahlke and Teschke in \cite{DaTe10} and further studied in \cite{DaHaTe12}; it is given by
\[ S = \left\{ \left( \begin{array}{cccccc} 1 & s_1 & s_2 & \ldots & \ldots & s_{d-1} \\ 
  & 1 & s_1 & s_2 &\ldots & s_{d-2} \\  &  & \ddots & \ddots & \ddots & \vdots \\ &  &  &\ddots & \ddots & s_2\\ & &  & & 1 & s_1 \\
 & & & & & 1 \end{array} \right) : s_1,\ldots,s_{d-1} \in \mathbb{R} \right\}˜˜.\]
 \end{enumerate} 
 For both types of groups, the basic facts of coorbit theory are already established, e.g. existence and well-definedness of coorbit spaces, as well as the existence of atomic decompositions in terms of bandlimited Schwartz wavelets. The existence of compactly supported atoms for the group in (i) is shown in \cite{DaStTe12}; for the Toeplitz shearlet group, this question has not yet been addressed, but will be settled below.  
\end{remark}
%
%
%

Shearing subgroups are closely related to abelian dilation groups, by the following observation:
\begin{proposition} \label{prop:char_shearsubs}
Let $S$ denote a closed, connected abelian subgroup of $T(d,\mathbb{R})$.  Then the following are equivalent:
\begin{enumerate}
\item[(i)] $S$ is the shearing subgroup of a generalized shearlet dilation group.
\item[(ii)] There is $\xi \in \mathbb{R}^d$ such that $S$ acts freely on $S^T \xi$ via the dual action, and in addition, ${\rm dim}(S) = d-1$. 
\item[(iii)] The matrix group $H = \{ rs: s \in S,  r \in \mathbb{R}^\times \}$ is an abelian irreducibly admissible dilation group. If we let $\mathfrak{h}$ and $\mathfrak{s}$ denote the Lie algebras of $H$ and $S$, respectively, then $\mathfrak{h}$ is an irreducible associative matrix Lie algebra over $\mathbb{R}$, and $\mathfrak{s}$ is the nilradical of $\mathfrak{h}$. 
\end{enumerate}
\end{proposition}

\begin{proof}
 For the implication (i) $\Rightarrow$ (ii), observe that if $S$ is the shearing subgroup of some shearlet dilation group $H$, the action of $H$ must have a free open dual orbit. To see this, we first note that there exists an open orbit with associated compact stabilizers. We next observe that the only nontrivial compact subgroup of $H$ is $\{ \pm {\rm id}_{\mathbb{R}^d} \}$: If $h \in H$ has a diagonal entry $\alpha$ with $|\alpha| \not= 1$, $h^n$ will have $\alpha^n$ as corresponding entry, for all $n \in \mathbb{Z}$, which shows that any subgroup containing $h$ will be non-compact. Thus any compact subgroup of $H$ is necessarily a compact subgroup of $S \cup (- S)$. But it is well-known that $\exp : \mathfrak{t} \to T(d,\mathbb{R})$ is bijective, and thus the same holds true for $\exp : \mathfrak{s} \to S$. In particular, since $S$ is abelian, it follows that $S$ is isomorphic to (the additive group of) a vector space, and thus has no nontrivial compact subgroups. 
 
 Thus the only available candidate for a nontrivial compact stabilizer is $\{ \pm  {\rm id}_{\mathbb{R}^d} \}$, but this is only contained in the stabilizer of the zero vector.  This shows that the action of $H$ on the open dual orbit is in fact free. 
 In particular, the dimension of $H$ is $d$, and since $S$ is of codimension one in $H$, we get ${\rm dim}(S) = d-1$. 

For the implication (ii) $\Rightarrow$ (iii), we note that if $S$ is of dimension $d-1$, and acts freely on $S^T \xi$, it follows that $S^T\xi$ has dimension $d-1$. Since the transposed action of $S$ leaves the first entry of each vector invariant, it follows that the $S^T(\xi + (t,0,\ldots,0)^T)$ has dimension $d-1$ as well, for any $t \in \mathbb{R}$. We may therefore in addition assume that the first entry of $\xi$ is nonzero. But then the dual orbit  $\left( \mathbb{R}^\times \times S \right)^T \xi$ must have dimension $d$, thus it is open. But for open orbits under an abelian matrix group the associated stabilizers are trivial, by \cite[Lemma 6]{Fu_abelian}, which implies that $H$ is irreducibly admissible.
Thus $H = \mathfrak{h}^\times$, where $\mathfrak{h}$ denotes the Lie algebra of $H$, which in this setting is also an associative algebra. If $\mathfrak{s}$ denotes the Lie algebra of $S$, then $\mathfrak{s} \subset \mathfrak{h}$ is a codimension one subspace consisting of nilpotent matrices, thus necessarily the nilradical. But then $H$ is irreducible over $\mathbb{R}$. 

For (iii) $\Rightarrow$ (i) we take the multiples of the identity as the diagonal complement to $S$.
\end{proof}

\begin{remark}
 As a consequence of the above proposition, we obtain a handy description of the image of the exponential map for shearing groups: If $S$ is a shearing subgroup with Lie algebra $\mathfrak{s}$, then we have
 \[
  \exp(\mathfrak{s}) = {\rm id}_{\mathbb{R}^d} + \mathfrak{s} = \{  {\rm id}_{\mathbb{R}^d} + X: X \in \mathfrak{s} \}~. 
 \] To prove the equality, denote the right-hand side by $S_0$. Then $S_0$ is a group: Closedness under products is clear, and closedness of the inverse is again seen by noting that the Neumann series breaks off after $d$ terms, and yields an element of ${\rm id}_{\mathbb{R}^d} + \mathfrak{s}$, since $\mathfrak{s}$ is an associative subalgebra.. Clearly, $S_0$ is a closed connected matrix group of dimension $d-1$, and in addition, we have $\exp(\mathfrak{s}) \subset S_0$, as a $d-1$-dimensional Lie-subgroup. But then equality follows. 
\end{remark}

The next result characterizes the one-parameter groups which may be used as complement to a shearing group, and provides an explicit description of the open dual orbit. 
\begin{proposition}
Let $S< {\rm GL}(d,\mathbb{R})$ denote a closed connected abelian group of upper triangular matrices of dimension $d-1$, and assume that there exists $\xi \in \mathbb{R}^d$ such that the dual stabilizer of $\xi$ in $S$ is trivial. Let $\mathfrak{s}$ denote the Lie algebra of $S$. Let $Y$ denote a nonzero diagonal matrix, and let $D := {\rm exp}(\mathbb{R}Y)$ the associated one-parameter group with infinitesimal generator $Y$. Then the following are equivalent:
\begin{enumerate}
\item[(i)] $H = DS \cup (-DS)$ is a shearlet dilation group.
\item[(ii)] For all $X \in \mathfrak{s}$ : $[X,Y] = XY-YX \in \mathfrak{s}$, and in addition the first diagonal entry of $Y$ is nonzero.
\end{enumerate}
For any choice of $Y$, $H^T$ acts freely on the open dual orbit $\mathcal{O} = \mathbb{R}^\times \times \mathbb{R}^{d-1}$. 
\end{proposition}
\begin{proof}
For the proof of (i) $\Rightarrow$ (ii), we note that by assumption, $H$ is a group, which implies in particular: ${\rm exp}(tY) s {\rm exp}(-tY) \in H$, for all $s \in S$ and $t \in \mathbb{R}$. 
Since $Y$ is diagonal and $s$ upper triangular, the diagonal entries of ${\rm exp}(tY) s {\rm exp}(-tY)$ are identically one, thus in fact ${\rm exp}(tY) s {\rm exp}(-tY) \in S$. Hence $D$ normalizes $S$, which implies $[Y,\mathfrak{s}] \subset \mathfrak{s}$ by differentiation. 

For the second condition on $Y$, we pick $\xi \in \mathbb{R}^d$ with the property that $H^T \xi$ is open (which exists by assumption). By \cite[Lemma 2]{Fu_abelian}, it follows that the map
\[
 \mathbb{R} \cdot Y + \mathfrak{s} \ni r Y + X \mapsto (rY+X)^T \xi
\] must be of rank $d$, hence bijective. Since $\mathfrak{s}$ consists of properly upper triangular matrices, it follows that
$\mathfrak{s}^T \xi \subset \{ 0 \} \times \mathbb{R}^{d-1}$. Assuming, in addition, that the first diagonal entry of $Y$ is zero, we get $(\mathbb{R} \cdot Y + \mathfrak{s})^T \xi \subset \{ 0 \} \times \mathbb{R}^{d-1}$, and the orbit of $\xi$ cannot be open. 

Conversely, assume (ii). Then $D$ normalizes $S$ and thus $DS$ is a subgroup of ${\rm GL}(d,\mathbb{R})$. One easily verifies that $DS$ is closed, and thus $H$ is closed. 

Our next aim is to show that $H$ is admissible, with open free dual orbit given by $\mathcal{O} = \mathbb{R}^\times \times \mathbb{R}^{d-1}$.  For this purpose we consider the associated abelian admissible matrix group $H_a = \mathbb{R}^\times S$ provided by Proposition \ref{prop:char_shearsubs} (ii) $\Leftrightarrow$ (iii). 

By Theorem \ref{thm:ab_dil} (i) $\Leftrightarrow$ (iii), there exists a $d$-dimensional commutative algebra $\mathcal{A}$ with unity and a linear isomorphism $\psi: \mathcal{A} \to \mathbb{R}^d$ such that 
\[ H_a = \{ h : \exists a \in \mathcal{A}^\times \mbox{ such that } \forall \xi \in \mathbb{R}^d,  h^T \xi = \psi(a \psi^{-1} (\xi)) \}. \]
Since $\mathcal{A}$ is abelian, the map $\rho: \mathcal{A} \ni a \mapsto (\xi \mapsto \psi(a \psi^{-1} (\xi)) )^T$ is an isomorphism $\mathcal{A} \to \mathfrak{h}_a$ of associative algebras: By choice, it is a group isomorphism $\mathcal{A}^\times \to \mathfrak{h}_a^\times = H_a$, and since the unit groups are open subsets of the respective algebras, they span them. But then $\rho$ must be at least onto, hence it is bijective for dimension reasons. In particular, $\mathcal{A}$ is an irreducible algebra, and its nilradical $\mathcal{N}$ has real codimension one in $\mathcal{A}$, and is mapped under $\rho$ bijectively onto $\mathfrak{s}$. 

For any $\xi = \psi(a) \in \mathbb{R}^d$, we have 
\[
 H_a^T \xi = \rho(\mathcal{A}^\times \psi(a))= \psi(\mathcal{A}^\times a)~,
\] showing that there is a unique open dual orbit of $H_a$, namely  $\psi(\mathcal{A}^\times)$. Furthermore, the action of $H_a$ on this orbit is free.  

We next show that $\psi(\mathcal{N}) = \{ 0 \} \times \mathbb{R}^{d-1}$. For this purpose, we note for any $a \in \mathcal{N}$ that $\psi(a) = \rho(a)^T \psi(1_\mathcal{A})$, and $\rho(a) \in \mathfrak{s}$ is a proper upper triangular matrix. Thus $\psi(\mathcal{N}) \subset \{ 0 \} \times \mathbb{R}^{d-1}$, and since $\psi$ is one-to-one and $\mathfrak{s}$ has dimension $d-1$, we obtain equality. 

It follows that $\psi(1_{\mathcal{A}}) \not\in  \{ 0 \} \times \mathbb{R}^{d-1}$, since otherwise $\psi$ could not be onto. Using that 
\[ \mathcal{A}^\times = \mathbb{R}^\times \cdot 1_\mathcal{A}+ \mathcal{N} ~,\]
this finally gives $\psi(\mathcal{A}^\times) = \mathbb{R}^\times \times  \mathbb{R}^{d-1}$. 

Before we turn to the group $H$, we need a slight refinement of this observation, namely
\begin{equation} \label{char:free_act}
 \xi \in \mathbb{R}^\times \times  \mathbb{R}^{d-1} \Leftrightarrow \mathfrak{s} \ni X \mapsto X^T \xi \mbox{ is injective. }
\end{equation}
Given $X \in \mathfrak{s}$ and $\xi \in \mathbb{R}^d$, we can write $X = \rho(a)$ and $\xi = \psi(b)$, with $a \in \mathcal{N}$ and $b \in \mathcal{A}$, and thus $X^T \xi = \rho(a)^T \psi(b) = \psi(a b) = 0$ iff $ab=0$.
Since $\xi \in \mathbb{R}^\times \times \mathbb{R}^d$ iff $b \in \mathcal{A}^\times$, this immediately settles ``$\Rightarrow$''. For the converse, assume that $\xi \not\in \mathbb{R}^\times \times  \mathbb{R}^{d-1}$, i.e., $b \not\in \mathcal{A}^\times$. If $b=0$, then $ab=0$ holds for all $a \in \mathcal{A}$. In the other case $b$ is nonzero and nilpotent, hence $a = b^j \not=0 \in \mathcal{N}$ exists with $ab = b^{j+1} = 0$, implying for $X = \rho(a) \not= 0$ that $X^T \xi = 0$. 

We can now finally consider the dual action of $H$. We start out by noting that $H^T \xi$ is open iff $\xi \in \mathbb{R}^\times \times  \mathbb{R}^{d-1}$. To see this, we again apply \cite[Lemma 2]{Fu_abelian} and consider the associated linear map
\[
 \mathbb{R} \times \mathfrak{s} \ni (r, X) \mapsto (rY  + X)^T \xi~,
\] which needs to be onto, thus bijective. By (\ref{char:free_act}), already the restriction to $\mathfrak{s}$ is only injective if $\xi \in \mathbb{R}^\times \times  \mathbb{R}^{d-1}$, showing the ``only-if'' part. For the other direction, we note that again by (\ref{char:free_act}), we know that the restriction to $\mathfrak{s}$ is one-to-one if $\xi \in \mathbb{R}^\times \times  \mathbb{R}^{d-1}$. In addition, the fact that $\mathfrak{s}$ consists of strictly upper triangular matrices implies that $\mathfrak{s}^T \xi \subset \{ 0 \} \times \mathbb{R}^{d-1}$, which then has to be an equality for dimension reasons. On the other hand, our assumptions on $Y$ and $\xi$ imply that the first entry of $Y^T \xi$ is nonzero, thus 
$Y^T \xi$ and $\mathfrak{s}^T \xi$ span all of $\mathbb{R}^d$. 

Now the arguments established so far show that the $DS$-orbit of any $\xi \in  \mathbb{R}^\times \times \mathbb{R}^{d-1}$ is open. Therefore, each of the two connected components of $\mathbb{R}^\times \times \mathbb{R}^{d-1}$ can contain at most one $DS$-orbit, thus must coincide with it. Since $H$ contains $\pm {\rm id}_{\mathbb{R}^d}$, it follows that $\mathbb{R}^\times \times \mathbb{R}^{d-1}$ is the unique open $H^T$-orbit.

Finally, we note that the action on the open orbit is free. If $({\rm exp}(rY) s)^T (\xi) = \xi$, for $\xi = (1,0,\ldots,0)^T$, $r \in \mathbb{R}^\times$ and $s \in S$, a comparison of the first entries of both sides, together with the fact that the first entry of $Y$ is nonzero, implies $r=0$, and thus $s^T \xi = \xi$. But we established that $S$ acts freely on $\mathbb{R}^\times \times \mathbb{R}^{d-1}$ (even the larger group $H_a$ does), and thus it follows that the stabilizer of $\xi$ is trivial. 
\end{proof}

\begin{remark} \label{rem:alg_sheardil}
 It is important to note that the correspondence between abelian shearlet dilation groups (modulo conjugacy) and {\em irreducible} commutative algebras over $\mathbb{R}$ (modulo isomorphisms) is again a bijection. More precisely, given any irreducible commutative algebra $\mathcal{A} = \mathbb{R} \cdot 1_\mathcal{A} + \mathcal{N}$ with nilradical $\mathcal{N}$, one can always choose a basis $Y_1,\ldots, Y_d$ with $Y_1 = 1_{\mathcal{A}}$ and such that, for all $k \in \{1,\ldots, d \}$, the subspace ${\rm span} \{ Y_k,Y_{k+1},\ldots, Y_d \}$ is an ideal in $\mathcal{A}$. Define a linear isomorphism $\psi : \mathcal{A} \to \mathbb{R}^d$ by $\psi(Y_i) = e_i$, the $i$th vector of the canonical basis of $\mathbb{R}^d$, and let $\rho: \mathcal{A} \to gl(\mathbb{R},d)$ denote the regular representation. Then the above-mentioned ideal property yields that
 $\rho(A)$ consists of lower triangular matrices, and thus $H = \rho(\mathcal{A}^\times )^T$ is an abelian shearlet dilation group. 
\end{remark}

As a consequence of the characterization of the open dual orbit associated to a shearing subgroup, we obtain two alternative characterizations of possibly independent interest. 
\begin{corollary}  \label{cor:char_shearsubs}
 Let $S< T(\mathbb{R},d)$ denote a closed connected abelian subgroup, with Lie algebra $\mathfrak{s}$. Then the following are equivalent:
 \begin{enumerate}
  \item[(a)] $S$ is a shearing subgroup.
  \item[(b)] ${\rm dim}(S) = d-1$, and the linear map sending each $X \in \mathfrak{s}$ to its first row is one-to-one.
  \item[(c)] There is a basis $X_2,\ldots,X_{d}$ of $\mathfrak{s}$ such that the first row of $X_i$ is the $i$th canonical basis vector of $d$-dimensional row space.
 \end{enumerate}
\end{corollary}
\begin{proof}
 (a) $\Rightarrow$ (b) follows from the proof of the previous proposition, since the map $X \mapsto X^T (1,0,\ldots,0)^T$ must be injective. (b) $\Rightarrow$ (c) is elementary linear algebra, and (c) $\Rightarrow$ (a) follows from the fact that assumption (c) implies that $S$ acts freely on the orbit of $(1,\ldots,0)$, and Proposition \ref{prop:char_shearsubs}.
\end{proof}

\begin{remark} \label{rem:shearlets}
Note the remarkable fact that the dual open orbit is independent of the precise choice of the shearlet group; the associated differential operator is given by ${\rm D}_{\mathcal{O}} = \frac{{\rm d}}{{\rm d}x_1}$.  As the next result shows, this independence extends to the admissibility condition.

The previous propositions suggest a procedure for the construction of shearlet dilation groups from commutative algebras. The starting point is a nilpotent algebra $\mathcal{N}$ of dimension $d-1$, which gives rise to the shearing subgroup and its Lie algebra. The second step then consists of determining the infinitesimal generators $Y$ of the diagonal complement. Since the associated one-parameter subgroups of $Y$ and $rY$ coincide, for any nonzero scalar $r$, we may normalize $Y$ to have $1$ as first diagonal entry. 
\end{remark} 

We next formulate the admissibility condition for the associated wavelet transform. Rather remarkably, this condition is the same for all shearlet dilation groups in a fixed dimension. 

\begin{theorem} \label{thm: shear_adm_cond}
Let $H < {\rm GL}(\mathbb{R}^d)$ denote a generalized shearlet dilation group. Then $\psi \in {\rm L}^2(\mathbb{R}^d)$ is admissible iff
\[
 \int_{\mathbb{R}^d} \frac{|\widehat{\psi}(\xi)|^2}{|\xi_1|^d} d \xi < \infty~. 
\]
\end{theorem}
\begin{proof}
 We employ the recipe described in \cite{BeTa,Fu96}; note that in the current setting, the dual orbit is free. We fix $\xi_0 = (1,0,\ldots,0)^T \in \mathbb{R}^d$, and define 
 \[ \Phi: \mathbb{R}^\times \times \mathbb{R}^{d-1} \to \mathbb{R}^+ ~,~ h^T \xi_0 \mapsto \Delta_H(h) |\det(h)|^{-1}~.\]
 This mapping is well-defined, see the discussion preceding equation (\ref{eqn:meas_orbit_Haar}). Now \cite[Theorem 13]{Fu96} yields that $\widehat{\psi} \in {\rm L}^2(\mathbb{R}^d)$ is admissible iff 
 \[
   \int_{\mathbb{R}^d} |\widehat{\psi}(\xi)|^2 \Phi(\xi) d \xi < \infty~.
 \]
 Thus it remains to compute $\Phi$. We first note that the shearing subgroup $S < H$ is an abelian normal subgroup, hence 
 $\Delta_H|_S \equiv 1$. The same holds for the determinant function. Hence for all $h = \exp(rY) s \in H$ with $r \in \mathbb{R}$ and $s \in S$, 
  we find
  \[ \Delta_H (\exp(rY) s)|{\rm det}(h s)|^{-1} = \Delta_H(\exp(rY)) |{\rm det}(\exp(r Y))|^{-1} ~.\] Thus we only need to compute $\Delta_H(h)$ and $\det(h)$ for $h = \exp(r Y)$. We assume that $Y$ is normalized such that its first diagonal entry equals one (recall Remark \ref{rem:shearlets}). We compute $\Delta_H(h)$ using the adjoint action of $H$ on the Lie algebra $\mathfrak{h} = \mathbb{R}  Y \oplus \mathfrak{s}$, with $\mathfrak{s}$ denoting the Lie algebra of the shearing subgroup (recall \cite[Lemma 2.30]{Folland_AHA}).  The conjugation action of $\exp(rY)$ on $\mathbb{R} Y$ is trivial, hence 
\[
 {\rm det}({\rm Ad})(\exp(-rY)) = {\rm det} (C_r) 
\] 
with the linear map
\[
  C_r : \mathfrak{s} \to \mathfrak{s} ~,~ X \mapsto \exp(-rY) X \exp(rY)~. 
\]
If we write $X = (X_{i,j})_{1 \le i,j \le d}$, we have by Corollary \ref{cor:char_shearsubs} that the mapping
$X \mapsto (X_{1,2},\ldots,X_{1,d})^T \in \mathbb{R}^{d-1}$ is a linear bijection. Hence we may compute the effect of conjugation with respect to these coordinates. If the diagonal entries of $Y$ are denoted by $y_1,\ldots,y_d$ and
$X' = \exp(-rY) X \exp(rY)$, then the coordinates of $X'$ are 
\[
 \left(X'_{1,2},\ldots,X'_{1,d} \right)^T = \left( \exp(r(y_2-y_1)) X_{1,2},\ldots, \exp(r(y_d-y_1)) X_{1,d} \right)^T~.
\] But this means that 
\begin{equation} \label{eqn:sh_Delta_H}
 \Delta_H(\exp(rY)) = \prod_{i=2}^d \exp(r(y_i-y_1)) = \exp(r({\rm trace}(Y)-d))~. 
\end{equation}
Note that the second equation used that $y_1 = 1$, by the normalization of $Y$ fixed above. 

On the other hand, we clearly have ${\rm det}(\exp(rY)) = \exp(r \cdot {\rm trace}(Y))$. 
 
 It remains to transfer these functions to the open dual orbit. For this purpose let $\xi = h^T \xi_0$ be given,
 where $h = \pm \exp(rY) s \in H$, with $r \in \mathbb{R}$ and $s \in S$. Assuming the above normalization of $Y$, and noting that the transposed action of $s$ leaves the first entry invariant, we find that $\exp(r) = |\xi_1|$. Thus, putting everything together, we find
 \[
  \Phi(\xi) = \Delta_H(\exp(rY)) |{\rm det}(\exp(rY))|^{-1} = e^{-rd} = |\xi_1|^{-d}~. 
 \]
\end{proof}

The close relationship between shearlet and abelian dilation groups allows to adapt the proof of Theorem \ref{thm:abelian_temp_embed} to the shearlet case.
Thus Theorem \ref{thm:main_cited} covers these groups as well, a fact which considerably extends the known results concerning their coorbit spaces. E.g., the existence of compactly supported elements in $\mathcal{B}_{v_0}$ was not previously known for the Toeplitz shearlet group. 

\begin{theorem} \label{thm:shear_temp_embed}
 Let $H$ denote a generalized shearlet dilation group, with $Y$ the infinitesimal generator of the diagonal subgroup, normalized to have one as first diagonal entry. Let $n$ denote the nilpotency class of the Lie algebra of the shearing subgroup. 
 Then $H$ fulfills the estimates (\ref{eqn:dec_est_norm})-(\ref{eqn:dec_est_A2}), with exponents
 \[
  e_2 = n-1+2\| Y \| ~,~ e_3 = |{\rm trace}(Y)|~,~ e_4 = |d-{\rm trace}(Y)|~.  
 \]
In particular, for given polynomial weight $w_0$ on $H$,  $q \in [1,\infty)$ and $s\ge0$, the open dual orbit is both $(s,q,w_0)$-temperately embedded and strongly $(s,w_0)$-temperately embedded, for suitable choices of indices.  
\end{theorem}
\begin{proof}
We employ the observations made in the previous results to simplify the calculations: For the representative of the open orbit, we may take $\xi_0 = (1,0,\ldots,0)^T$. Let $y_1,\ldots,y_d$ denote the diagonal entries of $Y$, we normalize $Y$ such that $y_1 = 1$. For ease of notation, we introduce $y_0= 0$.
We write $h \in H$ as $h = \pm ({\rm id}_{\mathbb{R}} + X) \exp(rY)$, with $r \in \mathbb{R}$ and $X \in \mathfrak{s}$. Then we obtain $h^T \xi_0 = \pm \exp(rY) (\xi_0 + X^T \xi_0)$, and we note that the two vectors in the bracket have disjoint supports. In fact, the distance of $h^T \xi_0$ to $\mathcal{O}^c$ can now be determined as $\exp(r)$, and a point in $\mathcal{O}^c$ of smallest euclidean distance to $h^T \xi$ is given by  $\pm \exp(rY) X^T \xi_0$.

We can therefore write
\[
 A_H(h) = \min \left( \frac{\exp(r)}{1+| \exp(rY) ( X^T \xi_0)|}, \frac{1}{1+|h^T \xi_0|} \right)~,
\]
in particular
\begin{equation} \label{eqn:AH_expr}
 A_H(h) \le \exp(r)~. 
\end{equation}

We can further estimate
\[
 | h^T \xi_0 |  \succeq |\exp(r)| + |\exp(rY) X^T \xi_0| \succeq |\exp(r)| + \min \{ \exp(r y_i) : 2 \le i \le d \} \| X \|
 \]
 where the first inequality used that the vectors $\xi_0$ and $\exp(rY) X^T \xi_0$ have disjoint supports, and the second estimate used that since the map $\mathfrak{s} \ni X \mapsto X^T \xi_0$ is one-to-one, we have the norm equivalence 
\[
 \| X \| \preceq |X^T \xi_0| \preceq \| X \|~.
\]
This yields on one hand
\begin{equation} \label{eqn:AH_nX}
 A_H(h) \| X \| \preceq  \max \{ \exp(-r y_i) : 2 \le i \le d \}~,
\end{equation}
and on the other 
\begin{equation} \label{eqn:AH_expmr}
 A_H(h) \preceq \exp(-r). 
\end{equation}

Since ${\rm det}(h) = \exp(r {\rm trace}(Y))$ and $\Delta_H(h) = \exp(r ({\rm trace}(Y)-d))$ (see the proof of Theorem \ref{thm: shear_adm_cond} for the latter), equations 
(\ref{eqn:AH_expr}) and (\ref{eqn:AH_expmr}) immediately imply (\ref{eqn:dec_est_det}) and (\ref{eqn:dec_est_A2}),
with $e_3 =  |{\rm trace}(Y)|$ and $e_4 = |d-{\rm trace}(Y)|$. 

The norm of $h$ can be estimated by 
\[
 \| h \| \le \| {\rm id} + X \| \max \{ \exp(ry_i) : 1 \le i \le d \}  \le (1+\| X \|) \max \{ \exp(ry_i) : 1 \le i \le d \}~.
\]

Putting all together (using $y_0=0$), we obtain
\begin{eqnarray*}
\| h \| A_H(h)^{n-1+2\|Y \|} &\preceq & \max \{ \exp(-r y_i) : 0 \le i \le d \}   \max \{ \exp(ry_i) : 1 \le i \le d \} A_H(h)^{n-2+ 2\| Y \|}  \\ & \preceq & 1~,
\end{eqnarray*} where the first inequality used (\ref{eqn:AH_nX}), and the second one used (\ref{eqn:AH_expr}) and
(\ref{eqn:AH_expmr}); note that $\| Y \|$ is the maximal modulus of the $y_i$. Also, note that since $\mathfrak{s}$ has positive dimension, $n \ge 2$. 

The norm of $h^{-1} = \exp(-rY) ({\rm id}+X)^{-1}$ can be estimated by
\[
 \|h^{-1} \| \le \| (1+X)^{-1}\| \max \{ \exp(-ry_i) : 1 \le i \le d \} \preceq (1+\| X\|)^{n-1} \max \{ \exp(-ry_i) : 1 \le i \le d \}
\] where we again used a Neumann series expression for the inverse to estimate the norm of $\| (1+X)^{-1}\| $; this series breaks off after $n-1$ terms. 
Thus we obtain 
\begin{eqnarray*}
\| h^{-1} \| A_H(h)^{n-1+2\|Y \|} & \preceq & \max \{ \exp(-r y_i) : 1 \le i \le d \}   \max \{ \exp(ry_i) : 1 \le i \le d \} A_H(h)^{ 2\| Y \|} \\
& \preceq &  1~,
\end{eqnarray*}
by analogous reasoning. 
\end{proof}

\begin{remark}
 We expect that for concrete choices of shearlet dilation groups the constants $e_1$  through $e_4$, and the associated estimates for the indices $\ell$ giving the required numbers of vanishing moments, have room for improvement. Note however that the constants in the theorem are largely independent of the precise choice of shearing group. 
\end{remark}

The following result specializes the theorem to the coorbit spaces ${\rm Co}({\rm L}^p(G))$, emphasizing that our results yield concrete, explicitly computable criteria. 
\begin{corollary}
 Assume that  $H< {\rm GL}(\mathbb{R}^d)$ is a generalized shearlet dilation group, and let $Y$
 denote the infinitesimal generator of the diagonal subgroup, normalized to have first entry equal to one.
 Let $n$ denote the nilpotency class of the Lie algebra of the shearing subgroup. 
  Fix $1 \le p \le \infty$, and let $\psi \in {\rm L}^2(\mathbb{R}^d)$ denote a function with 
  vanishing moments in $\mathbb{R} \times \{  0 \} \subset \mathbb{R}^d$ of order $r$ and $|\widehat{\psi}|_{r,r} < \infty$,
  where 
  \begin{equation} \label{eqn:nvan_shearlet}
   r = d(1+2n) + \lfloor 4 \| Y \|(d+1) + \frac{3}{2} |{\rm trace}(Y)| + |d-{\rm trace(Y)}| \rfloor~. 
  \end{equation}
Then $\psi$ is an atom for the coorbit space ${\rm Co}({\rm L}^p(G))$.

Compactly supported atoms $\psi$ can be constructed by picking $f \in C_c(G)$ with continuous derivatives of order up to $2r$, and letting $\psi = \frac{{\rm d}^r}{{\rm d} x_1^r} f$. 
\end{corollary}
\begin{proof}
Recall that we can take $v_0(x,h) = w_0(h) = {\rm max}(1,\Delta_G(h))$, and $s=0$. 
 In order to use the formula from Theorem \ref{thm:crit_str_temp_embed}, we only
  need an estimate for $e_1$.
  For this purpose, we recall the notations and observations from the proof of Theorem \ref{thm:shear_temp_embed}, i.e., 
  $h = \pm ({\rm id}_{\mathbb{R}} + X) \exp(rY)$ for an element $h \in H$. We then have by (\ref{eqn:sh_Delta_H}), that 
 \[
  \Delta_G(0,h) = \exp(rd)~,
 \] and since $A_H(h) \preceq \max(\exp(r),\exp(-r))$, we obtain
 \[
  \Delta_G(0,h) A_H(h)^d \preceq 1~,
 \] i.e. $e_1 = d$. 
 Now Theorem \ref{thm:crit_str_temp_embed} yields that the dual orbit is strongly temperately embedded with index
 \begin{eqnarray*}
  \ell  & = &\lfloor e_1 + e_2(2d+2) + \frac{3}{2} e_3 +  e_4 \rfloor + d+1 \\
   & = & \lfloor d+  (n-1+2\| Y \|) (2d+2) + \frac{3}{2} |{\rm trace}(Y)| + |d-{\rm trace}(Y)|
   \rfloor + d + 1 
 \end{eqnarray*}
Thus Theorem \ref{thm:main_cited}(c) yields that any function $\psi$ with vanishing moments of order $r= d+1+ \ell$, and with $|\widehat{\psi}|_{r,r}< \infty$, is in $\mathcal{B}_{v_0}$. 
Note here that $d,n \in \mathbb{N}$ allows to simplify $r$ to the right-hand side of (\ref{eqn:nvan_shearlet}). 
\end{proof}

\begin{remark}
 For the standard higher-dimensional shearlets from Example \ref{rem:ex_shearlet}(a), one typically uses the complement with infinitesimal generator $Y = {\rm diag}(1,\frac{1}{2},\ldots,\frac{1}{2})$, which yields $\| Y \|  = 1$ as well as ${\rm trace}(Y) = (d+1)/2$. The nilpotency class is $n=2$, and we finally obtain $r = 10d+4+\left\lfloor \frac{d+1}{4} \right\rfloor$. 
For $d=2$, this gives $r=24$, somewhat worse than the value obtained by more explicit methods in Remark \ref{rem:dim_two}. 

For the Toeplitz shearlet case from Example \ref{rem:ex_shearlet}{b}, typically used with $Y = {\rm diag}(1,1,\ldots,1)$, we get $\| Y \|=1$, ${\rm trace}(Y) = d$ and
nilpotency classe $n=d$, which results in $r = 2d^2 + 6d + 4 + \left\lfloor \frac{d}{2} \right\rfloor$. Note the quadratic term, as a result of the nilpotency degree $n=d$.  
\end{remark}


\subsection{Shearing groups in dimensions three and four}

The correspondence between shearing groups and commutative associative algebras allows a systematic construction of shearing groups. For small dimensions, it is possible to give a complete list. 

We first note that the two classes of examples given in Remark \ref{rem:ex_shearlet} correspond to two extreme cases of irreducible algebras of dimension $d$: One quickly realizes that the algebras underlying the two examples in Remark \ref{rem:ex_shearlet} have nilpotency classes $2$ and $d$, respectively. Furthermore, it is easy to see that for each  $n \in \{ 2, d \}$, there exists precisely one $d$-dimensional algebra
$\mathcal{A}$ with $n(\mathcal{A}) = n$. For the case $n=2$, one has $ab = 0$ for any $a,b \in \mathcal{N}$, and it is clear that for two such algebras, any linear isomorphism is an algebra isomorphism as well. For the case $n=d$, there exists $a \in \mathcal{N}$ with $a^{d-1} \not= 0$, and this implies that $1_{\mathcal{A}}, a,a^2, \ldots, a^{d-1}$ is a basis of $\mathcal{A}$. But then the map $\mathbb{R}[X]/(X^d) \to \mathcal{A}$, $X \mapsto a$ is an algebra isomorphism. 

Thus Remark \ref{rem:ex_shearlet} provides all possible shearing groups in dimension three up to conjugacy.  

We next give a complete list (up to conjugacy) of all shearing groups in dimension 4. 
\begin{itemize}
 \item $n(\mathcal{A}) = 2,4$ provides the two groups from Remark \ref{rem:ex_shearlet}.
 \item For $n(\mathcal{A}) = 3$, it is proved in \cite{Fu_abelian} that every four-dimensional commutative algebra with unit and nilpotency class $3$ is isomorphic to $\mathbb{R} [X,Y]
 / (X^3,Y^2-aX^2,XY)$, for a unique $a \in \{ -1, 0, 1 \}$. Picking the basis $Z_1 = 1_{\mathcal{A}}, Z_2 = X, Z_3 = Y, Z_4 = X^2$ and following the construction programme described in Remark \ref{rem:alg_sheardil} of computing the regular representation via the linear map $\psi: Z_i \mapsto e_i$ results in the following groups:
 \[
  H_a = \left\{ \left( \begin{array}{cccc} s & t_1 & t_2 & t_3 \\ 0 & s & 0 & t_1 \\ 0 & 0 & s & a t_2 \\ 0 & 0 & 0 & s 
                       \end{array} \right) : s \in \mathbb{R} \setminus \{ 0 \}~,t_1,t_2,t_3 \in \mathbb{R} \right\}~, a \in \{-1,0,1 \}~.
 \] 
\end{itemize}

%
%

\section*{Concluding remarks}

The main goal of this paper was to provide easily accessible criteria for atoms, that are valid for large classes of dilation groups. This unified treatment is a fairly novel feature of the theory; prior to \cite{Fu_coorbit,Fu_atom}, most sources concerned with coorbit theory for higher-dimensional wavelet transforms concentrated on special cases \cite{DaKuStTe,DaStTe10,DaHaTe12,DaStTe12,Ul12,FePa}, each of which was treated with tailor-made approaches. It would be interesting to study further extensions of the method, e.g. to the setting of Besov spaces on symmetric cones \cite{Ch13}, or to wavelet systems arising from the mock metaplectic representation \cite{DMDV13}. For some systems of the latter type, which are closely related to shearlets, this has already been studied in \cite{dahlke2014different}. 

We finish by emphasizing the central role of the dual action in all our considerations. Note that the usefulness of this action is not just restricted to the characterization of analyzing vectors and atoms. Another recent development in wavelet coorbit theory that heavily relies on the dual action is the paper \cite{FuVo}, which embeds the theory of coorbit spaces into the context of  decomposition spaces. The latter class of spaces was introduced by Feichtinger and Gr\"obner \cite{DecompositionSpaces1}, and has recently attracted renewed attention, for example in the context of shearlet smoothness spaces \cite{BorupNielsenDecomposition,Labate_et_al_Shearlet}. The language of decomposition spaces provides a unified framework for the treatment of many different types of smoothness spaces, including $\alpha$-modulation spaces (and thus the classes of modulation and inhomogeneous Besov spaces), but also shearlet smoothness spaces. By the results in \cite{FuVo} {\em all} wavelet coorbit spaces of the type 
discussed in this paper, at least for weights that are independent of the translation parameter (corresponding to the case $s=0$), also fall into the category of decomposition spaces. This result provides a method to systematically study embeddings between coorbit spaces associated to {\em different} dilation groups. 

\section*{Acknowledgements}

This research was in part funded by the Excellence
Initiative of the German federal and state governments, and by the German Research Foundation (DFG),
under the contract FU 402/5-1. We thank Felix Voigtlaender for many helpful comments, and the referees for useful suggestions and additional references.  The second author was supported by a grant from Ferdowsi Universty of Mashhad, No. MP93308RRT.

\bibliography{coorbit_abelian_shearlet.bib}

\begin{thebibliography}{10}

\bibitem{BeTa}
David Bernier and Keith~F. Taylor.
\newblock Wavelets from square-integrable representations.
\newblock {\em SIAM J. Math. Anal.}, 27(2):594--608, 1996.

\bibitem{BorupNielsenDecomposition}
L.~Borup and M.~Nielsen.
\newblock Frame decomposition of decomposition spaces.
\newblock {\em J. Fourier Anal. Appl.}, 13(1):39--70, 2007.

\bibitem{Ch13}
Jens~Gerlach Christensen.
\newblock Atomic decompositions of {B}esov spaces related to symmetric cones.
\newblock In {\em Geometric analysis and integral geometry}, volume 598 of {\em
  Contemp. Math.}, pages 97--110. Amer. Math. Soc., Providence, RI, 2013.

\bibitem{CzaKi12}
Wojciech Czaja and Emily~J. King.
\newblock Isotropic shearlet analogs for {$L^2(\Bbb R^k)$} and localization
  operators.
\newblock {\em Numer. Funct. Anal. Optim.}, 33(7-9):872--905, 2012.

\bibitem{dahlke2014different}
Stefan Dahlke, Filippo De~Mari, Ernesto De~Vito, S{\"o}ren H{\"a}user, Gabriele
  Steidl, and Gerd Teschke.
\newblock Different faces of the shearlet group.
\newblock {\em J. Geom. Anal.}, 2015.
\newblock In press.

\bibitem{DaHaTe12}
Stephan Dahlke, S{\"o}ren H{\"a}user, and Gerd Teschke.
\newblock Coorbit space theory for the {T}oeplitz shearlet transform.
\newblock {\em Int. J. Wavelets Multiresolut. Inf. Process.}, 10(4):1250037,
  13, 2012.

\bibitem{DaKuStTe}
Stephan Dahlke, Gitta Kutyniok, Gabriele Steidl, and Gerd Teschke.
\newblock Shearlet coorbit spaces and associated {B}anach frames.
\newblock {\em Appl. Comput. Harmon. Anal.}, 27(2):195--214, 2009.

\bibitem{DaStTe10}
Stephan Dahlke, Gabriele Steidl, and Gerd Teschke.
\newblock The continuous shearlet transform in arbitrary space dimensions.
\newblock {\em J. Fourier Anal. Appl.}, 16(3):340--364, 2010.

\bibitem{DaStTe12}
Stephan Dahlke, Gabriele Steidl, and Gerd Teschke.
\newblock Multivariate shearlet transform, shearlet coorbit spaces and their
  structural properties.
\newblock In {\em Shearlets}, Appl. Numer. Harmon. Anal., pages 105--144.
  Birkh\"auser/Springer, New York, 2012.

\bibitem{DaTe10}
Stephan Dahlke and Gerald Teschke.
\newblock The continuous shearlet transform in higher dimensions: Variations of
  a theme.
\newblock In {\em Group {T}heory: {C}lasses, {R}epresentations and
  {C}onnections, and {A}pplications}, pages 165--175. Nova Science Publishers,
  2010.

\bibitem{DMDV13}
Filippo De~Mari and Ernesto De~Vito.
\newblock Admissible vectors for mock metaplectic representations.
\newblock {\em Appl. Comput. Harmon. Anal.}, 34(2):163--200, 2013.

\bibitem{DVJaPo}
Ronald~A. DeVore, Bj{\"o}rn Jawerth, and Vasil Popov.
\newblock Compression of wavelet decompositions.
\newblock {\em Amer. J. Math.}, 114(4):737--785, 1992.

\bibitem{FeiGr0}
Hans~G. Feichtinger and Karlheinz Gr{\"o}chenig.
\newblock A unified approach to atomic decompositions via integrable group
  representations.
\newblock In {\em Function spaces and applications ({L}und, 1986)}, volume 1302
  of {\em Lecture Notes in Math.}, pages 52--73. Springer, Berlin, 1988.

\bibitem{FeiGr1}
Hans~G. Feichtinger and Karlheinz Gr{\"o}chenig.
\newblock Banach spaces related to integrable group representations and their
  atomic decompositions. {I}.
\newblock {\em J. Funct. Anal.}, 86(2):307--340, 1989.

\bibitem{FeiGr2}
Hans~G. Feichtinger and Karlheinz Gr{\"o}chenig.
\newblock Banach spaces related to integrable group representations and their
  atomic decompositions. {II}.
\newblock {\em Monatsh. Math.}, 108(2-3):129--148, 1989.

\bibitem{DecompositionSpaces1}
H.G. Feichtinger and P.~Gr{\"o}bner.
\newblock {B}anach spaces of distributions defined by decomposition methods,
  {I}.
\newblock {\em Math. Nachr.}, 123(1):97--120, 1985.

\bibitem{FePa}
H.G. Feichtinger and M.~Pap.
\newblock Coorbit theory and {B}ergman spaces.
\newblock In {\em Harmonic and Complex Analysis and its Applications}, Trends
  in Mathematics, pages 231--259. Springer, New York, 2014.

\bibitem{Folland_AHA}
Gerald~B. Folland.
\newblock {\em A course in abstract harmonic analysis}.
\newblock Studies in Advanced Mathematics. CRC Press, Boca Raton, FL, 1995.

\bibitem{Fu96}
Hartmut F{\"u}hr.
\newblock Wavelet frames and admissibility in higher dimensions.
\newblock {\em J. Math. Phys.}, 37(12):6353--6366, 1996.

\bibitem{Fu_abelian}
Hartmut F{\"u}hr.
\newblock Continuous wavelet transforms with abelian dilation groups.
\newblock {\em J. Math. Phys.}, 39(8):3974--3986, 1998.

\bibitem{FuDiss}
Hartmut F\"uhr.
\newblock {\em Zur {K}onstruktion von {W}avelettransformationen in h\"oheren
  {D}imensionen}.
\newblock PhD thesis, TU M\"unchen, 1998.

\bibitem{Fu_LN}
Hartmut F{\"u}hr.
\newblock {\em Abstract harmonic analysis of continuous wavelet transforms},
  volume 1863 of {\em Lecture Notes in Mathematics}.
\newblock Springer-Verlag, Berlin, 2005.

\bibitem{Fu10}
Hartmut F{\"u}hr.
\newblock Generalized {C}alder\'on conditions and regular orbit spaces.
\newblock {\em Colloq. Math.}, 120(1):103--126, 2010.

\bibitem{Fu_coorbit}
Hartmut F{\"u}hr.
\newblock Coorbit spaces and wavelet coefficient decay over general dilation
  groups.
\newblock {\em Trans. AMS}, 2015.
\newblock In press.

\bibitem{Fu_atom}
Hartmut F{\"u}hr.
\newblock Vanishing moment conditions for wavelet atoms in higher dimensions.
\newblock {\em Adv. Comput. Math.}, 2015.
\newblock In press.

\bibitem{FuVo}
Hartmut F{\"u}hr and Felix Voigtlaender.
\newblock Coorbit spaces viewed as decomposition spaces.
\newblock {\em J. Funct. Anal.}, 2015.
\newblock In press.

\bibitem{Gr}
Karlheinz Gr{\"o}chenig.
\newblock Describing functions: atomic decompositions versus frames.
\newblock {\em Monatsh. Math.}, 112(1):1--42, 1991.

\bibitem{shearlet_book}
Gitta Kutyniok and Demetrio Labate, editors.
\newblock {\em Shearlets}.
\newblock Applied and Numerical Harmonic Analysis. Birkh\"auser/Springer, New
  York, 2012.
\newblock Multiscale analysis for multivariate data.

\bibitem{Labate_et_al_Shearlet}
D.~Labate, L.~Mantovani, and P.~Negi.
\newblock Shearlet smoothness spaces.
\newblock {\em J. Fourier Anal. Appl.}, 19(3):577--611, 2013.

\bibitem{Mu}
Romain Murenzi.
\newblock Wavelet transforms associated to the {$n$}-dimensional {E}uclidean
  group with dilations: signal in more than one dimension.
\newblock In {\em Wavelets ({M}arseille, 1987)}, Inverse Probl. Theoret.
  Imaging, pages 239--246. Springer, Berlin, 1989.

\bibitem{Ul12}
Tino Ullrich.
\newblock Continuous characterizations of {B}esov-{L}izorkin-{T}riebel spaces
  and new interpretations as coorbits.
\newblock {\em J. Funct. Spaces Appl.}, pages Art. ID 163213, 47, 2012.

\end{thebibliography}
\bibliographystyle{plain}
\end{document}